\documentclass[12pt,oneside]{amsart}
\usepackage{amscd,amsmath,amssymb,amsfonts,a4}
\usepackage{hyperref}

\usepackage{pstricks}
\usepackage{color}
 
\newcommand{\makered}[1]{}

\newlength{\rulebreite}

\def\timesover#1#2#3{\ \xymatrix@1@=0pt@M=0pt{ _{#1}&\times&_{#2} \\& ^{#3}&}\ }
\def\otimesover#1#2#3{\ \xymatrix@1@=0pt@M=0pt{ _{#1}&\otimes&_{#2} \\& ^{#3}&}\
}

\usepackage{amscd}
\usepackage{latexsym}
\usepackage[latin1]{inputenc}

\newtheorem{theorem}{Theorem}

\newtheorem{lemma}[theorem]{Lemma}

\newtheorem{proposition}[theorem]{Proposition}

\numberwithin{equation}{section}

\def\A{\mathbb{A}}

\def\N{\mathbb{N}}
\def\Z{\mathbb{Z}}
\def\PP{\mathbb{P}}
\def\Q{\mathbb{Q}}
\def\R{\mathbb{R}}
\def\C{\mathbb{C}}

\def\F{\mathbb{F}}

\def\V{\mathbb{V}}

\newcommand{\cf}{{\em cf. }}

\newcommand{\cF}{{\mathcal  F}}
\newcommand{\cG}{{\mathcal  G}}

\newcommand{\cO}{{\mathcal O}}

\newcommand{\Pic}{{\rm Pic\,}}

\newcommand{\Spec}{{\rm Spec\, }}

\renewcommand{\div}{{\rm div\,}}

\newcommand{\pr}{{\rm pr}}

\newcommand{\degtr}{{\rm deg.tr}}

\newcommand{\rk}{{\rm rk \,}}

\newcommand{\hE}{{\widehat{E}}}

\newcommand{\hM}{{\widehat{M}}}

\newcommand{\hT}{{\widehat{T}}}
\newcommand{\hV}{{\widehat{V}}}
\newcommand{\hX}{{\widehat{X}}}
\usepackage{amsmath}
\usepackage{amsfonts}
\usepackage{amssymb}

\newcommand{\hZ}{{\widehat{Z}}}

\newcommand{\ra}{\rightarrow}
\newcommand{\lrasim}{\stackrel{\sim}{\longrightarrow}}
\newcommand{\lra}{\longrightarrow}
\newcommand{\hra}{\hookrightarrow}

\newcommand{\hlra}{{\lhook\joinrel\longrightarrow}}

\newcommand{\Bun}{\mathbf{Bun}}

\makeindex 

\usepackage[all]{xy}
\theoremstyle{plain}
\newtheorem{thm}{Theorem}
\newtheorem{lem}[thm]{Lemma}
\newtheorem{cor}[thm]{Corollary}
\newtheorem{prop}[thm]{Proposition}
\theoremstyle{definition}
\newtheorem{defn}[thm]{Definition}

\newtheorem{rmk}[thm]{Remark}

\newtheorem{ex}[thm]{Example}

\newtheorem{claim}[thm]{Claim}
\numberwithin{thm}{section}
\numberwithin{equation}{section}

\newcommand{\ga}[2]{\begin{gather}\label{#1}#2 \end{gather}}

\newcommand{\surj}{\twoheadrightarrow}

\newcommand{\sD}{{\mathcal D}}
\newcommand{\sE}{{\mathcal E}}
\newcommand{\sF}{{\mathcal F}}
\newcommand{\sG}{{\mathcal G}}
\newcommand{\sH}{{\mathcal H}}

\newcommand{\sL}{{\mathcal L}}
\newcommand{\sM}{{\mathcal M}}

\newcommand{\sO}{{\mathcal O}}
\newcommand{\sP}{{\mathcal P}}

\newcommand{\sT}{{\mathcal T}}
\newcommand{\sU}{{\mathcal U}}
\newcommand{\sV}{{\mathcal V}}
\newcommand{\sW}{{\mathcal W}}

\renewcommand{\P}{{\mathbb P}}
\makeatletter\let\@wraptoccontribs\wraptoccontribs\makeatother
\begin{document}

\title[Simply connected varieties in characteristic $p>0$]{ Simply connected
varieties in characteristic $p>0$ }
\author{H\'el\`ene Esnault}
\address{
Freie Universit\"at Berlin, Arnimallee 3, 14195, Berlin,  Germany}
\email{esnault@math.fu-berlin.de}
\author{Vasudevan Srinivas}
\address{School of Mathematics, Tata Institute of Fundamental Research, Homi Bhabha Road, Colaba, Mumbai-400005, India} 
\email{srinivas@math.tifr.res.in}

\contrib[\kern3ex with an appendix \vbox to 1.05em{}by]{Jean-Beno\^it Bost}
\address{D\'epartement  de Math\'ematiques, Universit\'e Paris-Sud, B\^at.~425, 91405, Orsay, France}
\email{jean-benoit.bost@math.u-psud.fr}

\date{ May 5, 2014}
\thanks{The first  author is supported by  the Einstein program and the ERC
Advanced
Grant 226257, the second author is supported by the Humboldt Foundation, by a
J.C. Bose Fellowship, and by a membership at IAS, Princeton.}
\subjclass[2010]{11G99, 14B20, 14G17, 14G99, 14F35}
\keywords{Stratified bundles, fundamental groups, formal and algebraic geometry}
\begin{abstract} We show that there are no non-trivial stratified bundles over a
smooth simply connected quasi-projective 
variety over an algebraic closure of a finite field, if the variety admits a
normal projective compactification with 
boundary locus of codimension $\ge 2$.

  \end{abstract}
\maketitle
\section{Introduction}
On a smooth quasi-projective variety $X$ over the field $k$ of complex numbers,
the theorem of Mal\v cev-Grothendieck  (\cite{Mal40}, \cite{Gro70}) shows   that
 the 
\'etale fundamental group $\pi_1^{{\rm \acute{e}t}}(X)$ controls the regular
singular $\sO_X$-coherent $\sD_X$-modules, that is the regular singular  flat
connections:  they are trivial  if $\pi_1^{{\rm \acute{e}t}}(X)=0$. 
The aim of this article is to show some variant if the ground field $k$ is
algebraically closed of characteristic $p>0$.

\medskip

Over the field  $\C$ of complex numbers,  the proof is very easy, but it makes a
crucial use of the topological
fundamental group $\pi_1^{\rm top}(X)$ and the fact that it is finitely
generated. Indeed, a   complex  linear representation has values in $GL(r, A)$
for a ring $A$ of finite type over $\Z$,  
and if it is non-trivial, it remains non-trivial after specializing to some
closed point of $A$. 
 If $k$ has characteristic $p>0$, we no longer have this tool at
our disposal. 

\medskip

All we know is that the category of $\sO_X$-coherent $\sD_X$-modules
  is Tannakian, neutralized by a rational point on $X$,
and that its profinite completion is, according to dos Santos,  the \'etale
fundamental group (\cite[Cor.~12]{dSan07}). 

\medskip

Nonetheless, as conjectured by Gieseker \cite{Gie75}, the same theorem holds 
true under the extra assumption that $X$ is projective (\cite{EsnMeh10}):
triviality of
$\pi_1^{\rm \acute{e}t}(X)$ implies triviality of  $\sO_X$-coherent
$\sD_X$-modules.
\medskip

On the other hand, it is shown in \cite[Thm.1.1]{Kin14} that dos Santos' theorem
{\it loc. cit.}  may be refined in the following way. The category of
$\sO_X$-coherent $\sD_X$-modules 
has a full subcategory of regular singular ones. The
profinite completion of its Tannaka group is then the tame fundamental group.

This enables one to ask (see \cite[Questions~3.1]{Esn12}): 
\begin{itemize}
\item[i)] If  $\pi_1^{\rm \acute{e}t}(X)=0$, are all $\sO_X$-coherent
$\sD_X$-modules 
 trivial? 
\item[ii)] If $\pi_1^{\rm \acute{e}t, tame}(X)=0$, are all regular singular
$\sO_X$-coherent $\sD_X$-modules  trivial?
\end{itemize}

\medskip

In this article, we address Question i). Fundamental groups of quasi-projective 
non-projective  smooth varieties in characteristic $p>0$ are not well
undersood. 
We give in Section~\ref{sec:field} some non-trivial examples where we can obtain
some
reasonable structure.

\medskip

Our main theorem \ref{thm:mainthm} asserts that Question i) has a positive
answer  if $X$ has a normal compactification with boundary of codimension $\ge
2$, and if $k$ is an algebraic closure of the prime field $\F_p$. 

\medskip 

The arithmetic assumption on the ground field comes from the application of
Hrushovski's main theorem \cite[Cor.~1.2]{Hru04} and the fact that we cannot, 
in general, define a surjective specialization homomorphism on the \'etale 
fundamental group (we do not know, however, if this can be done for the smooth 
loci of a family of normal projective varieties). See
Section~\ref{sec:spe}.

\medskip

 The geometric assumption on the existence of a good compactification with small
boundary enables us to show a strong boundedness theorem \ref{thm:boundedness}.
The main issue, if we drop the assumption,  is to define a suitable family of 
extensions    to a particular 
normal projective compactification of $X$ of the $F$-divided  vector bundles 
associated to a 
$\sO_X$-coherent $\sD_X$-module via Frobenius descent, in such a way that they
form a bounded 
family of sheaves. 
 
 \medskip
 
 Under the geometric assumption in our result, the reflexive extension of the 
 bundles does it. See Remark~\ref{rmk:mainthm}.  In general, we do not know. 
 Over the field of complex numbers, if we assume in addition that the   
 $\sO_X$-coherent $\sD_X$-module
is regular singular, then we have Deligne's canonical extension at our disposal.
Indeed,
Deligne shows in \cite{Del14} boundedness for those. His non-algebraic proof
relies on the fact that the topological fundamental group is finitely
generated. 
  
  \medskip
  
  We use the existence of a projective ample complete intersection curve  to
reduce the problem to the case where the underlying $F$-divided bundles  $E_n$ of
a
$\sO_X$-coherent $\sD_X$-module 
are stable of slope $0$, and of given Hilbert polynomial.  This
comes from the Lefschetz theorem Theorem~\ref{thm:bost} for stratified bundles, 
which ultimately relies on Bost's most recent generalization of the
Grothendieck 
$\mathbf{Lef}$ conditions, see Appendix~\ref{app}.
Then the proof follows the line of the proof of the main theorem in 
\cite{EsnMeh10}.

\medskip

In Section~\ref{sec:field}, relying again on the Lefschetz theorem, we are able
to give a non-trivial family of examples with trivial \'etale fundamental group 
for which one can also conclude  that $\sO_X$-coherent $\sD_X$-modules 
 are trivial, this time
over any field, as we can argue without using Hrushovski's theorem {\it loc.
cit.}. 

\medskip

{\it Acknowledgements:}  It is a pleasure to thank Adrian Langer for
a discussion his boundedness theorems, which  play a
very important
role in this article. We also thank Alexander Schmitt for discussions.  We
thank 
Jean-Beno\^{\i}t Bost for his interest, and for contributing in the form of
Appendix A to 
this article.  We thank the referees for their precise work which helped us to
improve 
the exposition of our text, in  particular,  it led us to treat the Verschiebung
map in 
a clearer way in the proof of Theorem~\ref{thm:mainthm}.

\section{Boundedness} \label{sec:boundedness}

We fix the notations for this section. Let $X$ be a projective normal
irreducible variety of dimension $d\ge 1$ over an algebraically closed field $k$
of characteristic $p>0$. Let  $j: U\to X$ be the open embedding of the regular
locus. Recall that if $E$ is a vector bundle, i.e., a locally free coherent
sheaf, 
on $U$, then $j_*E$ is a coherent sheaf on $X$, and is the reflexive hull of
$E$. 
Indeed by normality,  one has 
$j_*j^*\sO_X=j_*\sO_U=\sO_X$,  which implies that 
$j_*j^*\sH om(\sF, \sO_X)=\sH om(\sF, j_*\sO_U) =\sH om(\sF, \sO_X)$  for any 
$\sO_X$- coherent sheaf $\sF$,  which is thus coherent. In addition,  $\sH
om(\sF, \sO_X)=
j_*j^*\sF^\vee$ if $\sF$ a vector bundle, and contains any other torsion free
extension  
to $X$ of $j_*\sF^\vee$. We apply this to $j^*\sF=E^\vee$.

We fix an ample Cartier divisor $Y\to X$.  For any coherent sheaf $\sE$ on $X$, 
we write 
$\sE(mY)$ for $\sE\otimes_{\sO_X} \sO_X(mY)$. For a vector bundle $E$ on $U$,
we define the Hilbert polynomial $\chi(X, j_*E(mY))\in \Q[m]$. Here $\chi$
stands 
for the Euler characteristic of the  coherent sheaf  $j_*E(mY)$, which is equal 
to $j_*(E(mY|_U))$ by the projection formula. We recall the definition of an 
$F$-divided sheaf in Definition~\ref{defn:Fdiv}. 

This section is  devoted to the proof of the following statement. 
\begin{thm} \label{thm:boundedness} We fix $r\in \N\setminus \{0\}$. 
There are finitely many polynomials $\chi_i(m)\in \Q[m], \  i \in I=\{1,\ldots,
N\}$ such that for any $F$-divided sheaf $E=(E_n,\sigma_n)_{n\in \N}$ of rank
$r$ 
on $U$, there is a $n_0(E)\in \N$ such that for all $n\ge n_0(E)$, 
$\chi(X, j_*E_n(mY))\in \{\chi_1(m), \ldots, \chi_N(m)\}$.
\end{thm}
If dim$(X)=1$,  then $X$ is smooth projective, so $E$ has degree $0$
(\cite[Cor.~2.2]{EsnMeh10}),  so one has 
$\chi(X, E(m))=r \chi(X, \sO_X(m))$ by Riemann-Roch on curves.
So we may assume $\dim X\geq 2$.

\medskip
In general, we reduce to the $ 2$-dimensional case as follows. Let $d={\rm
dim}(X)\ge 3$.  We choose a 
Noether normalization $h: X\to \mathbb{P}^d$ over $k$. 
This defines the open $\sU$ of the product of the dual projective spaces $
((\mathbb{P}^d)^\vee)^{d-2}$,   whose points   $x\in \sU$  parametrize
intersections of 
hyperplanes $H_x:= (H_{1,x} \cap \ldots \cap H_{d-2,x}) \to \mathbb{P}^d$ 
such that $h^*H_x=:Y_x\subset X\otimes_k \kappa(x)$ intersects   $X\setminus U$
properly.  Here $\kappa(x)$ is the residue field of $x$.  The open $\sU$ is
non-empty, and irreducible. 

\begin{lem} \label{lem:red_2dim}
Let $\Spec K\xrightarrow{x} \sU$ be a geometric generic point, $i_x: Y_x\to
X\otimes_k K$ be the corresponding closed embedding, $j_{Y_x}: V_x = (U
\otimes_k K) \cap Y_x\to Y_x$.  Then 
\begin{itemize}
\item[i)] $Y_x$ is a projective normal irreducible $2$-dimensional variety;
\item[ii)]  For any vector bundle $E$ on $U$, $i_x^*j_* E=  
(j_{Y_x})_* (i_x|_{U\otimes_k K})^* E $.
\end{itemize}
\end{lem}
\begin{proof}
For a vector bundle $E$, the sheaf  $h_*j_*E$  is  reflexive, thus is  locally
free outside of codimension $\ge 3$ as  $\mathbb{P}^d$  is smooth. 
So there is a non-empty open $\sU_E\subset \sU$ defined over $k$, such that for
any $x\in \sU_E$,
$Y_x$ is smooth in codimension $1$ and 
$h_*(j_*E|_{Y_x})= (h_*j_*E)|_{H_x}$  is  locally free. In particular,  $
(j_*E)|_{Y_x} =(j|_{Y_x})_*E|_{Y_x\cap U}$
 is  a reflexive sheaf on $Y_x$. 
Applied to $\sO_U$,  this yields $\sO_{Y_x}=(j|_{Y_x})_*\sO_U|_{Y_x\cap U}$, so
$Y_x$ is normal, by Serre's criterion (see \cite[Thm.~7]{Sei50} for the 
original proof). By Artin vanishing theorem \cite[Cor.~3.5]{SGA4}, $Y_x$ is 
connected. As $Y_x$ is normal, $Y_x$ is irreducible. This shows i). One does 
not need here $x$ to be generic, a closed point $x$ in $\sU_{\sO_U}$ also does
it. 
On the other hand, the $K$-point $x: \Spec K\to \sU$ factors through $x: \Spec
K\to \sU_E\cap
\sU_{\sO_U}\to \sU$ for any $E$. This implies ii).

\end{proof}

\begin{cor} \label{cor:red_2dim}
Theorem~\ref{thm:boundedness} is true if and only if it is true for $d=2$ over
any algebraically closed field $k$. 
\end{cor}
\begin{proof}
We apply A. Langer's boundedness theorem \cite[Thm.~4.4]{Lan04}:  let $X$ be a 
projective  variety  of dimension $\ge 3$ over an algebraically closed field
$k$,  
let $S\to X$ be a dimension $2$ complete intersection of  very ample 
divisors  $Y_i, \ 2\le i\le {\rm dim}(X)$,  of some linear system
$\sO_X(1)=\sO_X(Y)$. 
Then fixing some polynomial $\chi(m)\in \Q[m]$ and some positive number $\mu$,
there are finitely many polynomials $p_i(m)\in \Q[m]$ such that 
the Hilbert polynomial  $\chi(X, E(m))$ of any pure sheaf $E$ on $X$ verifying  
$\chi(S, E|_S(m))=\chi(m)$ and $\mu_{{\rm max}}(E)\le \mu$ 
is equal to one of the $p_i(m)$. Here $\mu_{{\rm max}}$ 
denotes the slope of a maximal destabilizing subsheaf (lowest term in the 
Harder-Narasimhan filtration) of $E$.  So it remains to show: there exits
a 
constant $\mu$ (depending only on $(X,\sO_X(1))$ and the rank), such that for
any
given $(E_n)_n$ as in Theorem~\ref{thm:boundedness}, $\{\mu_{\rm max}(E_n)\}$
is 
bounded above by $\mu$ for all but finitely many $n$.

\medskip

 Given a constant $\mu \in \R$, in order to show that $\mu_{{\rm
max}}(E_n)\le \mu$  
for all $n\ge n_0$ (for some $n_0$ depending on $\{E_n\}$),  it is enough to do
it
after base change $X_K\to X$ over an algebraically closed field $K\supset k$.

\medskip

 As in the situation of Lemma~\ref{lem:red_2dim}, we may consider the open
$\sU'$ in the product 
$((\P^d)^{\vee})^{d-1}$ corresponding to complete intersection curves in $X$ of
the form $\cap _{i=1}^{d-1}Y_i$, 
where $Y_i\in  |\sO_X(1)|, s=1,\ldots, d$. Let $\eta$ be the generic point of
$\sU'$, and $C_{\eta}$ be the corresponding 
curve in $X_{\eta}$. Then $C_{\eta}$ is a projective nonsingular curve contained
entirely in $U_{\eta}$. 
(See e.g. \cite[Thm.~6.10]{Jou83}).  We choose $K$ to be an algebraic closure of
$k(\eta)=k(\sU')$.
Note that the resulting $k$-morphisms $C_{\eta}\to X$ and $C_{K}\to X$ are flat,
since they are obtained 
from the flat morphism from $I\to X$, where $I\subset X\times_k\sU'$ is the
incidence variety. 

\medskip

Hence, if $E$ is any reflexive sheaf on $X$, any injective morphism of coherent
sheaves $E_1\to E$ remains so after 
pull-back to $C_K$, which we may view as obtained by first pulling back to
$X_K$, and then restricting to the closed 
subscheme $C_K\subset X_K$, which is a smooth complete intersection curve in
$X_K$. In particular, if $E_1\subset E$ 
is the maximal destabilizing subsheaf, then 
$E_1\mid_{C_K}\hookrightarrow E\mid_{C_{K}}$ is a destabilizing subsheaf
(perhaps not maximal) of the 
same degree. By definition of $\mu_{\rm max}$, it is thus enough to show 
that $\{ \mu_{\rm max}(E_n|_{C_K})\}$ is bounded above by some $\mu$, depending
only on $X$, after possibly
omitting finitely many values of $n$. 

\medskip

Since for any torsion free sheaf $V$ on $C_K$, the Riemann-Roch theorem implies
that the degree of $V$ is 
bounded above by the dimension of $H^0(C_K, V)$ plus some constant depending
only on  $C_K$ and the rank 
of $V$,  thus only on $C_K$ and $r$ for $V\subset E_n|_{C_K}$, it  is enough to
bound above $H^0(C_K, V)$ by a 
constant $\mu$, for any such $V\subset E_n\mid_{C_K}$,   after possibly omitting
finitely many values of $n$.

\medskip

The maximal trivial sub $F$-divided sheaf of $(E_n|_{C_K})_n$ is by definition $(V \otimes \sO_{C_K})_n $ with $F^*(V\otimes \sO_{C_K})=V\otimes \sO_{C_K}$, where $V$ is a finite dimensional $K$-vector space of rank $\le r$.
The decreasing sequence $(H^0( C_K, E_n|_{C_K}))_n$, where the inclusion 
$H^0(C_K,  E_n|_{C_K}) \to H^0(C_K,  E_{n-1}|_{C_K})$ is defined by the
pull-back by Frobenius,  thus 
 stabilizes to $V$,   for $n\ge n_0$ for some natural number $n_0$.    This finishes the proof. 

\end{proof}

\begin{proof}[Proof of Theorem~\ref{thm:boundedness}]
By Corollary~\ref{cor:red_2dim}, we may assume that $X$ has dimension $2$. We
use below 
various standard facts about normal projective surfaces; details may for example
be found
in the book \cite{Badescu}; for finite generation of the N\'eron-Severi group of
a surface, 
see \cite[V, Thm.3.25]{Milne}.

We denote by $i: Y\to X$ a smooth projective irreducible ample curve, by
$\Sigma=X\setminus U$ the singular locus of $X$, which thus consists of finitely
many closed points (outside $Y$).
Let $\pi:  X'\to  X$ be a
desingularization such that $\pi^{-1} \Sigma$ is a strict normal crossings
divisor $D=\cup_{x\in \Sigma} D_x, \ D_x=\cup_i D_{x,i}$ the irreducible
components. 
We may thus identify $U$ with $X'\setminus D$.
We note that
$NS_{\Sigma}(X'):=\oplus_{x,i} \Z\cdot D_{x,i}\subset NS(X'), $
and the intersection form of $NS(X')$ has negative definite restriction to this
subgroup.  
We set
$NS(X)={\rm coker}\, \left(\oplus_{x,i} \Z\cdot D_{x,i}  \to NS(X')\right) $
so that $NS(X)$ may be identified with the N\'eron-Severi group of $X$ 
(of Weil divisor classes modulo algebraic equivalence). This is a finitely 
generated abelian group; further, the orthogonal complement
$NS_{\Sigma}(X')^{\perp}$ of $NS_{\Sigma}(X')$ naturally maps to $NS(X)$, such
that
$NS_{\Sigma}(X')^{\perp}\to NS(X)$
is injective with finite cokernel. We also note that there is a natural,
surjective homomorphism
$ {\rm Pic}(U) \to NS(X) $
induced by the surjection 
${\rm Pic}(X') \to {\rm Pic}(U) $
obtained by restriction of line bundles.

\medskip

For each bundle $V^U$ on $U$, we define 
$  V=j_*V^U, \  V'= (\pi^* V)^{\vee \vee}.$
Then $V$ is a coherent reflexive sheaf on $X$, and $V'$  is a vector bundle on
$X'$. 
For any vector bundle $V'$ on $X'$, we define $c_1(V')\in NS(X')$ and
$c_2(V')\in \Z$
as the images of the corresponding algebraic Chern classes in $CH_*(X')$.

With those notations, it is enough to show 
\begin{itemize}
\item[i)] the subset  $\{c_1(V'_n)
\}\subset  NS(X')$  is finite when $\sV^U=(V^U_n, \sigma_n)$ runs through all
rank
$r$ $F$-divided  sheaves;
\item[ii)] 
there is an $M>0$ such that 
for each such  $\sV^U$, there is an $n_0(\sV^U)\ge 0$ such that for all  $n\ge
n_0(\sV^U)$, we have 
$|c_2(V'_n)|\le M$ in $\Z$. 
\end{itemize}
Indeed,  if  i) and  ii) hold, then  by the Riemann-Roch theorem on $X'$, 
there  is a finite set of polynomials $S\subset
\Q[m]$ such that
\ga{}{ {\rm for \ all} \  \  \sV^U,  \chi(X', V'_n(m\pi^*Y))\in S\subset \Q[m],
\  {\rm for \ all} \ n\ge n_0(\sV^U). \notag}
One has
\ga{}{ \chi(X', V'_n(m\pi^*Y))= \chi(X, \pi_* V'_n(m Y)) - {\rm
dim}_k H^0(X,R^1\pi_*V'_n) . \notag}
But $\pi_* V'_n=V_n$, and since $V_n$ is generated along $\Sigma$ by its
sections defined on a neighbourhood of 
$\Sigma$, so is $\pi^*V_n$ generated by its sections on a neighbourhood of
$\pi^{-1}(\Sigma)$. 
Thus there is an $\sO_{X'}$-linear homomorphism  $\oplus_1^{r+1} \sO_{X'}\to
V'_n$ defined in a neighbourhood of $D$, which is surjective outside of
codimension $\ge 2$.
Thus  one has a surjection 
\ga{}{\oplus_1^{r+1} R^1\pi_*\sO_{X'} \surj R^1\pi_*V'_n  \notag}
locally near $\Sigma$, and both sheaves are supported within $\Sigma$, which
bounds the 
dimension over $k$ of $H^0(X,R^1\pi_*V'_n)$.  This shows then
Theorem~\ref{thm:boundedness}, 
 assuming i) and ii); we now prove these. Note that our proof of boundedness of
$\dim_k H^0(X,R^1\pi_*V'_n)$
did not use i) and ii), and so we may make use of this in the proof of i) and
ii). 
 
\medskip

We denote by  $c_1(V)$ the image of $c_1(V')$ in $NS(X)$. 
For a rank $r$ $F$-divided sheaf $\sV^U=(V^U_n, \sigma_n)_{n\in \N}$ on $U$, we
first
note that the sequence $c_1(V_n)$ in the finitely generated abelian group 
$NS(X)$ satisfies
$ p^n  c_1(V_n)=c_1(V_0) \mbox{ for all $n\ge 0$},$
since a similar relation holds between the determinant line bundles of $V_n^U$
and $V_0^U$. 
This implies that $c_1(V_n)$ is torsion for all $n$. It also implies that for
some 
positive integer $\delta$ depending only on $(X,\Sigma)$, 
$\delta \cdot c_1(V'_n)\in NS_{\Sigma}(X')\subset NS(X')$. 
Hence, for any $V^U_n$, we may uniquely write 
$\delta \cdot c_1(V'_n)=\sum m_{x,i}(V'_n) D_{x,i}$ where
$m_{x,i} \in \Z,$
So the assertion for $c_1$ is equivalent to saying that {\em the $m_{x,i}(V'_n)$
are all bounded}.  

\medskip 

The matrix $(D_{x,i}\cdot D_{x,j})_{ij}$ is negative definite  for any $x$, so
boundedness of the $m_{x,i}(V'_n)$ is equivalent to the statement that the
subset 
$\{ c_1(V'_n)\cdot D_{x,i} , \ \forall x,i\} \subset \Z $   is 
bounded.
We adapt Langer's argument from \cite[Prop.~4.6]{Lan00}.
 
 We first write $D=\sum_{x,i} D_{x,i} $  also for the corresponding reduced
divisor.
Now, consider the exact sequence
 \[0\to V'_n\to V'_n(D)\to V'_n(D)|_D\to 0.\]
Since $\pi_*V'_n=V_n$ is reflexive, so that $\pi_*V'_n=\pi_*V'_n(D)$, we deduce
that the boundary homomorphism 
\ga{}{ H^0(D,  V'_n(D)|_D) \to  H^0(X, R^1\pi_*V'_n) \notag}
is injective.  We had seen earlier that
\[\dim H^0(X, R^1\pi_*V'_n) \leq (r+1)\dim H^0(X,R^1\pi_*\sO_{X'}).\]
Thus 
\ga{}{ \chi(V'_n(D)|_D)\le h^0(V'_n(D)|_D ) \le C_1 \notag}
for a constant  $C_1>0 $ depending only on $X$ and $r$.

\medskip

On the other hand,  
${\rm deg} (V'_n|_{D_{x,i}}):=c_1(V'_n)\cdot D_{x,i}\ge 0$ for all $x,i$. 
Indeed, on a neighbourhood of $D$, $V'_n$ is spanned by $(r+1)$ global section
outside of codimension $\ge 2$, and
 ${\rm det}(V'_n)$ has a non-zero section on this neighbourhood, the divisor of
which  intersects $D$ in 
 dimension $\leq 0$.  One thus has
\ga{20}{ C_1 -rD^2 -   r\chi(\sO_D)  \\ 
 \ge  \chi(V'_n(D)|_D)  -rD^2 -   r\chi(\sO_D) 
   = \chi(V'_n|_D)-r\chi(\sO_D) = {\rm deg}(V'_n|_D) \ge  0 . \notag }
Since $ {\rm deg}(V'_n|_D)=\sum_{x,i} {\rm det}(V'_n)\cdot D_{x,i}$
where each term ${\rm det}(V'_n)\cdot D_{x,i}  \ge 0 $,  
we conclude that there is a constant $C_2>0$ such
that
\ga{}{ 0\le c_1(V'_n)\cdot D_{x,i}\le C_2  \  \forall x,i.\notag}
Thus there is a constant  $C_3>0$ such that
\ga{}{ | m_{x,i}(V'_n)|\le C_3 \  \forall x,i. \notag }
This finishes the proof for
 $c_1$.

\medskip

We show the statement for $c_2$.  The isomorphism $ (F^n)^*V^U_n\to V^U_0$
extends to an injective $\sO_{X'}$-linear map
 $(F^n)^* V'_n\to  V'_0$.  Thus ${\rm det}(V_0')= 
 p^n {\rm det} ( V'_n) + A_n$ where $A_n=\sum a_{x,i}(n)D_{x,i}, \ a_{x,i}(n)\in
\N $.  Thus \eqref{20} implies 
 $ 0\le a_{x,i}(n)\le (p^n+1) C_3.$
 On the other hand, the ideal sheaf of $A_n$ on $X'$ annihilates
 $Q_n= V'_0/(F^n)^*V_n$. Thus one has  induced surjections
 \ga{}{ V'_0 \surj V'_0|_{A_n}, \ q_n: V'_0\mid_{A_n}\surj Q_n, \   \  
(F^n)^* V'_n \surj  K_n, \mbox{ with}\ K_n:={\rm Ker}(q_n). \notag}

\medskip

Let $\sF$ denote any of the sheaves $Q_n, K_n, V'_0|_{A_n}$, each of which 
is a quotient of a vector bundle of rank $r$ on $A_n$. All three sheaves 
are generated by their global sections outside a set of 
 dimension $\le 0$ supported in $D$. So one has maps 
 $\oplus_1^r \sO_{A_n} \to \sF $
with cokernel supported in dimension $\leq 0$, and 
thus 
\ga{}{ h^1(\sF) \le r h^1(\sO_{A_n}) \notag}
for each of the sheaves $\sF$, and further one has 
\ga{}{ H^0(X ,R^1\pi_*\sO_X)\surj H^1(\sO_{A_n}).\notag}

\medskip

So,  one obtains from Riemann-Roch for $\sO_{A_n}$ the existence of a constant
$C_4>0$ such that 
$h^0(\sO_{A_n})=   \chi(\sO_{A_n})+ h^1(\sO_{A_n}) \le p^{2n}C_4.$
One also has
$ h^0(V_0'|_{A_n})=\chi(V'_0|_{A_n}) +h^1(V'_0|_{A_n})={\rm
deg}(V'_0|_{A_n})+r\chi(\sO_{A_n}) + h^1(V'_0|_{A_n}) .$
On the other hand, one has the exact sequence
\ga{}{0\to H^0(K_n)\to H^0(V'_0|_{A_n}) \to H^0(Q_n)\to H^1(K_n).\notag}
Thus, using that ${\deg}(V'_0|_{D})$ is bounded,  there is a constant $C_5>0$
such that 
$h^0(\sF)\le p^{2n}C_5$
for any of the choices of $\sF$.
This,   for $\sF=Q_n$,  and boundedness of  $h^1(\sF)$ imply, via the
Riemann-Roch formula
$\chi(Q_n)=\frac{1}{2} c_1(X)\cdot ch_1(Q_n)
+ ch_2(Q_n),$  that 
here is a constant $C_6>0$ such that 
\ga{}{ | ch_2(Q_n)| \le p^{2n} C_6. \notag}
Here we use 
 the ``numerical Chern character'' $ch(Q_n)=ch_1(Q_n)+ch_2(Q_n) \in NS(X')\oplus
 \frac{1}{2}\Z $.
Now we use the definition of $Q_n$ to conclude
$|ch_2(V'_0)-p^{2n} ch_2(V'_n)| \le p^{2n} C_6.   $
Thus 
\ga{}{ | ch_2(V'_n)|\le \frac{1}{p^{2n}} |ch_2(V'_0)| + C_6.\notag}
This  shows boundedness for $ch_2(V_n)$. 
 The  statement for $c_1$ (i.e., finiteness for possible $c_1$) shows now  the
statement for $c_2$.  This finishes the proof of the theorem.

\end{proof}

\section{First main theorem} \label{sec:mainthm}
\begin{defn} \label{defn:Fdiv}
Let $C$ be a  connected  scheme of finite type defined over an algebraically
closed field $k$.  An $F$-divided sheaf
$\sE$ is a sequence $(E_0, E_1, \ldots, \sigma_0, \sigma_1,\ldots)$ of
$\sO_C$-coherent sheaves $E_n$ on $C$,  with  $\sO_C$-isomorphisms $\sigma_n:
E_n\to F^*E_{n+1}$, where $F$ is the absolute Frobenius of $C$. The category of
$F$-divided sheaves is constructed by defining Hom$(\sE, \sV)$ as usual: one
replaces  $E_n$ be its inverse image $E'_n$ on the $n$-th Frobenius twist
$C^{(n)}$ of $C$,  the $\sO_C$-isomorphism $\sigma_n$ by $\sigma'_n; E'_n\to
F_n^*E'_{n+1}$,  where $F_n: C^{(n)}\to C^{(n+1)}$, and then $f\in {\rm Hom}(E,
V)$ consists of $f_n\in {\rm Hom}_{C^{(n)}}(E'_n, V'_n)$ commuting with the
$\sigma'_n$.

\end{defn}
(We may abuse notation and write $\sE=(E_0,E_1,\ldots)$ to denote an $F$-divided
sheaf, 
suppressing the maps in the notation.) 

If $C$ is smooth, by Katz' theorem \cite[Thm.~1.3]{Gie75},  the category of
$F$-divided sheaves is equivalent to the category of $\sO$-coherent
$\sD$-modules.  The category is then
$k$-linear  Tannakian.

We note that, even without smoothness, an $F$ divided 
sheaf $\sE=(E_0,E_1,\ldots)$ has component sheaves $E_i$ which are locally
free. 
This follows easily from an argument with Fitting ideals, as in
\cite[Lem.~6]{dSan07} 
 (regularity is not needed for the argument there, attributed to
Shepherd-Barron). 

\medskip 

The aim of this section is to prove the 
\begin{thm}  \label{thm:mainthm}
Let $X$ be a normal projective variety defined over $k=\bar \F_p$.  Let $U$ be
the regular locus.  If $\pi_1^{\rm \acute{e}t}(U)=0$, then there are no
non-trivial stratified bundles on $U$.

\end{thm}

\begin{prop} \label{prop:Fdiv}
Let $C$ be as in Definition~\ref{defn:Fdiv}.
\begin{itemize}
\item[i)] The category of $F$-divided sheaves on $C$ is $k$-linear, Tannakian.
\item[ii)] If $f: C'\to C$ is a universal homeomorphism, i.e. an integral
surjective and radical morphism,  then $f^*$ induces an equivalence of
categories between $F$-divided sheaves on $C$ and on $C'$. 

\end{itemize}

\end{prop}
\begin{proof}
We prove ii).
A  fixed power of the absolute Frobenius $F_C $ of $C$ factors through $C'$, so
$(F_C)^N: C \to  C'\xrightarrow{f} C$.  Given $(V_0, V_1, \ldots, \tau_0,
\tau_1,\ldots)$ on $C$, one defines $E_n=f^*V_n, \sigma_n=f^*\tau_n$.

We prove i). 
We refer to   \cite{Saa72}  when $C$ is smooth. 
By ii), we  may assume that $C$ is reduced.

We  show  that the category is abelian.  To this aim, we first assume that $C$ 
is irreducible. Let $u: (V_0, V_1,\ldots)\to (W_0, W_1,\ldots)$ be a  morphism, 
and let $r$ be the generic rank of the image. Set 
$L_n={\rm Im} \ \bigwedge^r V_n \subset  P_n=\bigwedge^r W_n$; 
then $L_n$ is a non-zero subsheaf of rank 1 of the vector bundle $P_n$. 

To say that $L_0\subset P_0$ is a subbundle (which is equivalent to ${\rm
im}\,V_0\to W_0$ 
being a subbundle of rank $r$), is equivalent to saying that for any point $x\in
C$,  with 
local ring $\sO_x$, maximal ideal $\mathfrak{m}_x$, and residue field
$k(x)=\sO_X/\mathfrak{m}_x$, 
we have that the map on fibres $L_x\otimes _{\sO_x}k(x)\to
P_x\otimes_{\sO_x}k(x)$ 
is non-trivial. If it was trivial, then for all $n$, $ (L_n)_x\subset {\rm Im}
(\mathfrak{m}_x\otimes (P_n)_x)$
thus $(L_0)_x\subset {\rm Im} (\,\mathfrak{m}_x^{p^n} \otimes_{\sO_x} (P_0)_x)$
for arbitrarily large 
$n$, which is impossible. Hence ${\rm Im}(\,V_0\to W_0)$ is a sub-bundle, and
similarly for $V_n\to W_n$ for any $n$.  
(Note that this is similar to the argument, alluded to above, in
\cite[Lem.~6]{dSan07}, 
using Fitting ideals.)

So the image of $u$ is a sub $F$-divided sheaf while restricted 
to the components of $C$. In particular, at the intersections points of two
components, the ranks are
equal. Since $C$ is connected, all the ranks are equal. This shows that the
category is abelian. 

We show that the category of $F$-divided sheaves is $k$-linear, that is that the
Homs are finite dimensional
$k$-vector spaces. 
It is enough to show that if $i:C_{{\rm reg}}\hookrightarrow C$ is the inclusion
of the smooth locus, 
the restriction homomorphism  ${\rm Hom}_C(\sV,\sW)\to {\rm Hom}_{C_{{\rm
reg}}}(i^*\sV,i^*\sW)$
is injective, which is true as it is already injective in the vector 
bundle category.  This finishes the proof.

\end{proof}

Choosing a rational point $c\in C(k)$ defines  a neutralization $E\mapsto
E_0|_c$
of the category, and thus a Tannaka group scheme $\pi^{\rm strat}(C)$ over $k$.
(We drop the base point $c$ from the notations).

\medskip

We now compare this definition with the  definition of stratified modules given
by Saavedra, 
see~\cite[Chap.~VI,~1.2]{Saa72} and references in there.  As this definition is
not used elsewhere in the paper, 
we use the terminology of ``stratified module in the sense of Saavedra'' to
refer to it.

\begin{prop} \label{prop:Saavedra}  An 
$F$-divided sheaf $\sE=(E_0, E_1,\ldots)$
 has the property that  for all $i$, $E_i$ 
is a
stratified module in the sense of 
Saavedra.
\end{prop} 
\begin{proof} it is enough to show it for $E_0$. 
Let $C(n)= C^n_{\Delta}$ be the formal completion of $C^n$ along the diagonal.
Then 
$C(n)=\varinjlim_{m\ge 1} C(n)_{m}= \varprojlim_{m \ge 0} C(n)_{p^m}$ where $_m$
means modulo the $m$-th power of the  ideal of the diagonal. The Frobenius $F^m:
 C^n\to C^n$ which on functions raises a function to its $p^m$-th power, yields
a 
factorization $F^m:  C\to  C(n)_{p^m} \xrightarrow{\Phi(n)_{p^m}} C.$
So if $\sE$ is $F$-divided in the sense of Proposition~\ref{prop:Fdiv}, one
defines the value of $\sE$  on $C(n)_{p^m}$
as $(\Phi(n)_{p^m})^*E_m$. 

\end{proof}
The  normality assumption on $X$  in Theorem~\ref{thm:mainthm}  is reflected in
the following Theorem~\ref{thm:bost}.
In order to prove Theorem~\ref{thm:mainthm}, we need only the case when $C$ is
smooth, which is much easier, but we will use the full strength of
Theorem~\ref{thm:bost}  in Section~\ref{sec:field}.

\begin{thm} \label{thm:bost}

Let $U$ be a smooth quasi-projective variety  of dimension $d\ge 1$ over an
algebraically closed field $k$, let $\iota: C \hookrightarrow U$ be a projective
 curve, which is a complete intersection of $(d-1)$-hypersurface sections. 
Then 
$\iota_*: \pi^{\rm strat}(C) \to \pi_1^{\rm strat}(U)$ is surjective.

\end{thm}
\begin{proof} 
By \cite[Prop.~2.21]{DM}, one has to show that $\iota^*$ is fully faithful and
that 
if $\sV=(V_n)_n $ is a $F$-divided sheaf  on $U$, and $\sW =(W_n)_n \subset
\iota^*\sV$ is an $F$-divided sheaf on $C$, then $\sW=\iota^*\tilde{\sW}$
for some sub-$F$-divided sheaf $\tilde{\sW}\subset \sV$.  Full faithfulness
follows from the unicity up to isomorphism of 
$\tilde{\sW}\subset \sV$ with the property that $\iota^*$ of this inclusion is 
 $\sW\subset \iota^*\sV$.

By definition of a $F$-divided sheaf as a crystal in the infinitesimal site of
$C$ , the inclusion $\sW\subset \iota^*\sV$ lifts to $\hat U_C$, the completion 
of $U$ along $C$. By an improved $\mathbf{Lef}$ condition, as formulated and
proven in the Appendix~\ref{app}, Theorem~\ref{subquot}, it
 implies that each subbundle $W_n\hookrightarrow \iota^* V_n$ lifts to a
sub-bundle 
$\sW_n'\hookrightarrow j_n^*V_n$ on some non-trivial open
$j_n:U_n\hookrightarrow U$ which contains 
$C$. Thus $U\setminus U_n$ consists of a closed subset of codimension $\ge 2$,
and $\sW_n'|_C=W_n$. 
Then, by the usual $\mathbf{Lef}$ condition (\cite[Exp.~10, Ex.~2.1
(1)]{SGA2}), 
 the section of $(F^* W_{n+1})^{\vee}\otimes W_n$ lifts to a  uniquely defined
section of 
$(F^*\sW'_{n+1})^{\vee}\otimes \sW'_n$ on the open  $U_{n+1}\cap U_n$  which 
is also the complement of a closed subset of codimension $\ge 2$ in $U$. 
We define  $ \widetilde{W}_n:=j_{n*}\sW_n'$. It is a torsion-free coherent
subsheaf of $V_n$. 
Moreover, since, as $U$ is smooth thus $F$ is faithfully flat,  $F^*j_*=j_*F^*$
for any open embedding $j:U\setminus \Sigma \to U$,
where $\Sigma$ is any closed subset (here taken to be $\Sigma=U\setminus U_n\cap
U_{n+1}$), 
we deduce that $\widetilde{W}_n\subset V_n$ defines a sub-$F$-divided sheaf (and
is
locally free). 
\end{proof}

\begin{lem} \label{lem:tr_fncs}
Let $X$ be a projective normal variety defined over an algebraically closed
field $k$, let
$j: U\to X$ be the open embedding of its regular locus.

\begin{itemize}
\item[i)]  If $\sE=(E_0, E_1, \ldots, \sigma_0, \sigma_1,\ldots)$ and 
$\sE'=(E'_0, E'_1, \ldots, \sigma'_0, \sigma'_1,\ldots)$ are two stratified
bundles on $U$ with isomorphic underlying vector bundles $E_n, E'_n$ for all
$n$, then
the two $F$-divided sheaves are isomorphic. In particular, $\sE$ is trivial if
and
only if the vector bundles $E_n$ are trivial.
\item[ii)] Let $\mathbb{I}=(\sO, \sO, \ldots, {\rm id}, {\rm id}, \ldots)$ be
the trivial object on $U$. Then the Ext group in the category of stratified
bundles fulfills:
${\rm Ext}^1(\mathbb{I}, \mathbb{I})\cong H^1_{\rm \acute{e}t}(U,
\Z/p)\otimes_{\F_p} k$.

\end{itemize}

\end{lem}
\begin{proof} We prove i). 
Following \cite[Prop.~1.7]{Gie75}, one just has to see that for any vector
bundle $V$ on $U$, $H^0(U,V)$ is a finite dimensional $k$-vector space, which is
fulfilled as $H^0(U,V ) = H^0(X, j_*V)$ and $j_*V$ is a coherent sheaf.  
Then ${\rm Hom}_U(E_n, E'_n)$ satisfies the Mittag-Leffler condition, 
and the proof given in   \cite[Prop.~1.7]{Gie75}  applies. 

sections
Then
over $k$.

We prove ii).  By i) and  \cite[(9)]{dSan07}, one has ${\rm Ext}^1(\mathbb{I},
\mathbb{I})\subset H^1_{\rm \acute{e}t}(U, \Z/p) \otimes k$. On the other hand, 
a
$\Z/p$-torsor $h: V\to U$ defines the stratified bundle $h_*\sO_V$, which is a
successive extension of $\mathbb{I}$ by $\mathbb{I}$. The bottom  sub of rank
$2$ defines a class in $ {\rm Ext}^1(\mathbb{I}, \mathbb{I})$ with image $h$.
\end{proof}

\begin{proof}[Proof of Theorem~\ref{thm:mainthm}]
Let $\iota: C\to U$ be a curve which is a  complete intersection of ($\dim
X-1)$ 
hyperplane sections $Y_i$ for the same very ample linear system  $|Y|$ on
$X$
of  very high degree. As argued in the proof of Corollary~\ref{cor:red_2dim},
$C$ is  contained in $U$, 
and is smooth by \cite[Thm.~1.1]{Poo04} (which applies also over a finite field
and to 
quasi-projective varieties.)

\medskip

Let $\sE$ be a non-trivial stratified bundle on $U$.
By \cite[Prop.~2.3]{EsnMeh10} and \cite{EsnMeh10b},  there is a $n_0 \ge 0$ such
that the stratified bundle $$(\iota^* E_{n-n_0},
\iota^*\sigma_{n-n_0})_{(n-n_0)\ge 0}$$ is a successive extension of stratified
bundles $(U_n, \sigma_n)_{n\ge 0}$ with the property that all $U_n$ are
$\mu$-stable bundles of slope $0$ on $C$. By Theorem~\ref{thm:bost}, there are
$(V_n, \tau_n)_{n \ge 0}$ on $U$ such that  $(U_n, \sigma_n)_{n\ge 0}=\iota^*
(V_n, \tau_n)_{n \ge 0}$, and those are irreducible objects.  Then by
Lemma~\ref{lem:tr_fncs} ii), if $(V_n, \tau_n)$ is trivial, so is $(E_n,
\sigma_n)$. So we may assume that $\sE$ is irreducible and that $\iota^*E_n$ is
$\mu$-stable for all $n\ge 0$.

By  Theorem~\ref{thm:boundedness}, there are finitely many polynomials
$\chi_1(m), \ldots, \chi_N(m)\in \Q[m]$ describing the Hilbert polynomials
$\chi(X, j_*E_n(mY ))$.  Let $\sP\subset \Q[m]$ denote this set of polynomials.
Let
 $M$ be the disjoint union of the moduli $M_i=M(\chi_i), i=1,\ldots, N$ of
$\chi_i$-stable 
 (in fact the moduli points of interest for us are even $\mu$-stable, thus
$\chi_i$-stable, 
 see \cite[Lem.~1.2.13]{HuyLeh97}) torsion-free coherent sheaves on $X$, of 
Hilbert polynomial
$\chi_i(m)$ \cite[Thm.~0.2]{Lan04}. Recall that $M(\chi_i)$  uniformly
corepresents the functor 
from the category of schemes over $k$ to the one of sets, which assigns to any
$T/k$ the set $\sM_X(\chi_i)(T)$ 
of isomorphism classes of coherent sheaves $E$ on $X\times_k T$, flat over $T$,
which are pure  
and $\chi_i$-stable on all geometric fibers of the projection $X\times_k T\to
T$.  It is a 
quasi-projective scheme. The sheaves $j_*E_n$ define moduli points $[j_*E_n] \in
M(k)$.  Let $   N_t
\subset M$ be the reduced subschemes defined as the Zariski closure of the
subset of closed points $ [j_*E_n], n\ge t$.  By the noetherian property,  
the decreasing sequence of subschemes $N_t$ stabilizes to $N$ say.

\medskip

Let  us denote by $\sigma: \Spec k \to \Spec k$ the Frobenius homomorphism of
the 
ground field $k$; by $X^{(1)} =X\otimes_{\sigma} k$,  $U^{(1)}=U\times_{\sigma}k$, $Y^{(1)}=Y\otimes_{\sigma} k$, 
$C^{(1)}=C\otimes_{\sigma} k$ the Frobenius twists of $X$, $U$, $Y$ and $C$ respectively;  by $W: X^{(1)} \to X$
the base change morphism 
over $\sigma$,  that is the canonical projection, which is an  isomorphism of
schemes;  by  $j^{(1)}: U^{(1)}\to X^{(1)}$ 
the open embedding, which is the base change by $\sigma$ of $j$.
One has 
\ga{}{ j^{(1)}_* W|_{U^{(1)}}^*E_n=W^*j_* E_n, \notag}
 thus  $  \chi(X, j_*E_n(mY )) =  \chi(X^{(1)}, j^{(1)}_*
W|_{U^{(1)}}^*E_n(mY^{(1)} ))$, and  
 ${\rm deg}(E_n|_C) ={\rm deg}(W^*E_n)|_{C^{(1)}}$. So stability is preserved. 
Let  $M_i^{(1)}$ and $M^{(1)}$ be  the corresponding moduli schemes.  Then 
$W^*$ is a transformation 
from the  moduli functor for coherent sheaves  on $X$ to the one for coherent 
sheaves  on $X^{(1)}$, which is invertible. Thus it defines a morphism $M_i\to
M_i^{(1)}$ over $\sigma$,
which is an isomorphism of schemes, in particular a bijection on $k$-points.  
One defines $N_t^{(1)}$ and $N^{(1)}$ associated to the moduli points 
$[W^*j_*E_n], \ n\in \N$, as reduced subschemes. Then 
$W^*: N\to N^{(1)}$   induces an isomorphism of schemes. We denote by
$(W^*)^{-1}: N^{(1)} \to N$ its inverse.

Our next goal is to construct a rational map $\Phi:N^{(1)} \dashrightarrow N$ 
which is defined on a neighbourhood of a dense set of the closed points $[j_*E_n]$ as 
above, and such that $\Phi([W^*j_*E_n])=[j_*F^*E_n]$.

If we construct a rational map $\tilde{\Phi}:N^{(1)}\dashrightarrow  M=\coprod M_i$ which is a morphism 
on a neighbourhood of a dense set of the closed points $[j_*E_n]$ as above, and satisfies 
$\Phi([W^*j_*E_n])=[j_*F^*E_n]$, then in fact $\tilde{\Phi}$ factors through $N$, and 
has dense image in $N$. 

\medskip 

If $\sH$ is a coherent sheaf on a product variety $T\times_k X$, one sets $\sH_t=\sH \otimes_{(T\times_k X)} (t\times_k X)$  for the pull-back coherent sheaf for any morphism of schemes  $t\to T$.
 We use in the sequel a variant of \cite[Prop.~1.1]{Har80}:
 \begin{lem} \label{lem:hart}
 Let $X$ be a projective normal variety over $k$,  $T$ be a smooth 
quasi-projective variety over $ k$,  $\sH$ be a reflexive sheaf 
 over $T\times_k X$. Then the locus of points $t\in T$ such that $\sH_t$ is
reflexive is a  constructible subset of $T$. 
 \end{lem}
 \begin{proof} 
 On a normal  quasi-projective scheme $Z$, a coherent sheaf $\sE$ is reflexive
if and only if there is an exact sequence $0\to \sE\to \sL\to \sF\to 0$ where
$\sL$ is locally free and $\sF$ is torsion free. [To obtain this sequence, one
uses quasi-projectivity to write an exact sequence 
 $\sL'^\vee \to \sL^\vee \to \sE^\vee \to 0$, where $^\vee$ denotes the dual,
and where $\sL, \sL'$ are locally free, then normality of $Z$ and reflexivity of
$\sE$ to obtain the  desired exact sequence using that the natural morphism
$\sE\to (\sE^\vee)^\vee$ is an isomorphism. 
  Conversely, if  one has such an exact sequence, $\sE$ is torsion free and the
cokernel of $\sE\to (\sE^\vee)^\vee$ lies in $\sF$, thus has to be zero  as $\sF$ is
torsion free.]

Applying this to $Z=T\times_k X$ and $Z=t\times_k X$, we see
that 
$\sH_t$ is reflexive if and only if the sheaves $\sH_t$ and $\sH'_t$ on
$t\times_k X$ are torsion free. Indeed, if $\sH'_t$ is torsion free, then the
criterion implies that ${\rm Ker}( \sL_t \to \sH'_t)$ is 
reflexive. If in addition $\sH_t$ is torsion free, then the surjective morphism
$\sH_t \to {\rm Ker} (\sL_t \to \sH'_t)$ is injective, thus an isomorphism.
Vice-versa, if $\sH_t$ is reflexive, then it is 
torsion free thus $\sH_t\to \sL_t$ has to be injective, and thus the surjective
morphism
$\sH_t \to {\rm Ker} (\sL_t \to \sH'_t)$ has to be an isomorphism, thus one has
an exact sequence
$0\to \sH_t\to \sL_t\to \sH'_t\to 0$, and taking again double duals, one
concludes that $\sH'_t\to ((\sH'_t)^\vee)^\vee$ is injective, thus $\sH'_t$ is
torsion free.
On the other hand, the locus of points $t\in T$ for which $\sH_t$ and $\sH'_t$
are torsion free is constructible (\cite[Prop.~9.4.8]{EGAIV}).

 \end{proof}

Restricting the Quot-scheme construction  \cite[Thms.~4.3.3,~4.3.4]{HuyLeh97} 
of $M$   to $N$ yields the existence of a 
quasi-projective scheme $T$ over $k$, a $\sG:=PGL(n)$ free action on  $T$, for
some  natural number $n$, such that $N$ 
is the categorial quotient  $N=T/\sG$ of $T$ by $\sG$,  together with a 
$\tilde{\sG}=GL(n)$-linearized coherent sheaf
$\sE$ on $T\times_k X$,   flat over $T$, such that for all geometric points
$t\in T$, the  point $t\in N$ is the moduli 
point of the restriction $\sE_t$ of $\sE$ to $t\times_k X$, which is stable and
of Hilbert polynomial $\chi$.

 \medskip
 We define the sheaf $\sE'=(1\times j)_* (1\times j)^* \sE$. Then $\sE'$ is
coherent, $\tilde{\sG}$-linearized and one has a 
 $\tilde{\sG}$-morphism $\rho: \sE\to \sE'$. For a subscheme $T'\subset T$, we
denote by $\rho_{T'}:  \sE_{T'}\to \sE'_{T'}$ the restriction of $\rho$  
 to $T'\times_k X$.
 
 \begin{claim} \label{claim1}
 There is a smooth dense open $\sG$-invariant  subscheme $T'\subset T$ 
 such that the restriction $\rho_{T'}$
 of $\rho$  is an isomorphism, and such that  for all geometric points $t\in
T'$, the sheaf $\sE_t|_{t\times_k U}$ is locally free and
$\rho_{t}:\sE_t\to \sE'_t$ and $\sE_t\to j_{t*}j_t^*\sE_t$ are isomorphisms.
 
 \end{claim}
 \begin{proof}  The connected components of the smooth locus of $T$ are 
 clearly $\sG$-stable, and so we may replace $T$ by any connected component of 
 its smooth locus, for the purposes of this claim. So we may assume here that 
 $T$ is smooth and connected.


Let $S$ denote the (dense) set of points $[j_*E_n]\in N$, and let $S^1=S\times_NT$. 
For any closed point $s\in S^1$, the local ring $\sO_{s,T}$ is a regular local ring,
so that the ideal sheaf of the closed subscheme $\{s\}\times X\subset T\times X$ is everywhere locally 
generated by a regular sequence. Further, $X$ is normal, and $\sE_s=j_*j^*\sE_s$, where $j^*\sE_s$ is locally free, 
and so $\sE_s$ satisfies Serre's property  $(S_2)$. By repeatedly applying \cite[Prop.~5.12.2]{EGAIV}
we conclude that $\sE$ satisfies Serre's property $(S_2)$ on $T\times X$ at all points of $\{s\}\times X$, 
for each $s\in S^1$. Hence there is an open subscheme $T\times_k X$ containing $S^1$ on which $\sE$ satisfies 
$(S_2)$ (namely, the open where $\sE\to (\sE^{\vee})^{\vee}$ is an isomorphism).
This is $\sG$-invariant, and has a maximal open subset of the form 
$T^1\times X$ with $T^1\subset T$ an open subscheme containing $S^1$; clearly $T^1$ is also $\sG$-invariant.

 On $T^1\times_k X$, $\rho_{T^1}:\sE_{T^1}\to \sE'_{T^1}=(1_{T^1}\times j)_*(1_{T^1}\times j)^*\sE_{T^1}$ is an isomorphism,
 since $\sE_{T_1}$ atisfies $(S_2)$ and $\rho_{T^1}$ is an isomorphism on a dense open subscheme with complement of 
 codimension $\ge 2$.  We slightly abuse notations, and 
 write $\sE_{T^1}=\sE'_{T^1}$.  The set $S^1$ lies in $ T^1$ and is dense. 
 
 On the other hand, for all $s\in S^1$, $\sE|_{s\times_k U}$ 
 is locally free. Thus the largest open subscheme of $T^2\times_k U$ on which
$\sE$ is locally free is not empty. It is $\sG$-invariant. 
 Its projection to $T^2$ is a constructible subset, which  contains the dense
subset $S^1$.  Thus it contains a dense open subscheme.  
 The maximal such open  $T^2$ is $\sG$-invariant, dense in $T^1$, and by
definition, $\sE|_{t\times_k U}$ is locally free for all geometric 
$t\in T^2$. 
 
Finally we apply  Lemma~\ref{lem:hart}, to conclude that there is a constructible subset $T^3\subset T^2$, 
which is $\sG$-invariant, and contains the dense subset $S^1$, consisting of the points $t\in T^2$ such that $\sE_t$ 
is reflexive. The largest open subset $T'\subset T^3$ is then also $\sG$-invariant, and $S^1\cap T'$ is dense in it.
For $t\in T'$, we then have that $\sE_t\to j_{t*}j_t^*\sE_t$ is an isomorphism, since $\sE_t$ is reflexive.
 
 This finishes the proof.
  \end{proof}
 
We abuse notations, denote by $T$ this open, by $S\subset T$ the
dense set of points, and simply by $\sE\to \sE'$ the isomorphism. So $T/\sG$ is
a dense open subscheme of $N$, in which  $S\cap (T/\sG)$ is dense.

 \medskip
 
 We perform the Quot-scheme construction on $X^{(1)}$, $N^{(1)}$,
 defining the smooth $\sG$-invariant scheme $T^{(1)}$,  with free $\sG$-action,
such that  $T^{(1)}/\sG \hookrightarrow N^{(1)}$ is open dense,  contained in the smooth locus, contains
 a dense open set of $k$-points from $S^{(1)}$,  and  defining the reflexive
sheaf $\sE^{(1)}$ on $T^{(1)}\times_k X^{(1)}$, locally free on $T^{(1)}\times_k
U^{(1)}$, flat over $T^{(1)}$,  $\chi$-stable and reflexive  on closed fibers.
  We denote by $F_{/k}: X\to X^{(1)}$ the relative Frobenius morphism.

 \begin{claim} \label{claim2} For each connected component $T_0$ of $T$, 
 there are a polynomial $\chi' (m)\in \Q[m]$,  with $\chi'(m)\in \sP$,
 and a smooth dense open $\sG$-invariant  subscheme  $ \sT^{(1)}\subset
T_0^{(1)}$, such that $S^{(1)}\times_{N^{(1)}} \sT^{(1)} \subset \sT^{(1)}$ is
dense, and   such that for all geometric points $t\in \sT^{(1)}$, $(1_t\times
j)_*(1_t\times F_{/k})^*(1_t\times j^{(1)})^*\sE^{(1)}_t$ is $\chi'$-stable.
 
  \end{claim}
 \begin{proof}
 
 We define  $\sV=(1_{T_0}\times j)_* (1_{T_0}\times F_{/k})^* (1_{T_0}\times j)^*\sE^{(1)}$,
which is  $\tilde{\sG}$-invariant, and apply for $\sV$ a similar argument as for $\sE$.
We conclude that there is a maximal  smooth open subscheme  $\sT^{(1)}$ of
$T_0^{(1)}$, which is $\sG$-invariant, on which $\sV$  is  $
 \sT^{(1)}$-flat, thus of constant Hilbert polynomial  $\chi'(m)\in \Q[m]$ on 
geometric fibers, and stable and reflexive on geometric fibers, of value
 $(1_t\times j)_*(1_t\times F_{/k})^*(1_t\times j^{(1)})^*\sE^{(1)}_t$. And by
construction, 
 $S^{(1)}\times_{N^{(1)}} \sT^{(1)} \subset \sT^{(1)}$ is dense.
  \end{proof}
  
 Let $M(\chi')$ be the moduli  scheme of $\chi'$-stable sheaves on $X$.  The
sheaf $\sV$ induces a morphism $\sT^{(1)} \to M(\chi')$ which sends a geometric
point $t$ to the moduli point of $(1_t\times j)_*(1_t\times F_{/k})^*(1_t\times
j^{(1)})^*\sE^{(1)}_t$.  This morphism is $\sG$-invariant, thus induces a
morphism  $\sU_0^{(1)}=\sT^{(1)}/\sG  \to  M(\chi')$. 
 Let $\sU_0=(W^*)^{-1}(\sU_0^{(1)})$, and let $\sU$ be the (disjoint) union of the $\sU_0$. 
Then $\sU\subset N$ is a dense open subscheme contained in the smooth locus of $N$. 
There is thus a well-defined morphism $\sU^{(1)}\to N$ and it may be viewed as a rational
dominant map $\Phi: N^{(1)}\dashrightarrow N$, called the  Verschiebung  map
(see \cite[App.~A]{Oss06}).

   \medskip

Let $\Gamma'$ be the  closure of the graph of $\Phi$ in $\sU^{(1)}\times_k \sU$.
Then   
$\Gamma' \subset \sU^{(1)} \times_k \sU $  is a closed reduced 
subscheme, which is birational to  $\sU^{(1)} $ via the first 
projection, and   dominates  $\sU$  via the second projection.  
We define 
$\Gamma  \subset \sU \times_k \sU $ 
by base 
change via  $(W^*)^{-1}$ on $\sU^{(1)}$.
This defines an  isomorphism of schemes $\Gamma \to \Gamma'$, and $\Gamma$ is
birational to $\sU$ via the first projection and dominant onto $\sU$ via the
second projection.  Thus $\Gamma$ is a rational dominant map.
 As $\sU$ has finitely 
many components, a power of this rational map stabilizes  them. 
 We apply \cite[Thm.~3.14]{EsnMeh10}, an application of 
Hrushovski's theorem \cite[Cor.~1.2]{Hru04}, to conclude that 
$\sU$ contains $k$-points $[j_*E]$, for $[j_*E]\in \sU(k)$,  such that
$(F^a)^*E=E$ for some $a\in \N\setminus \{0\}$, and $E$ defines a $F$-divided
sheaf. 
By \cite[Satz~1.4]{LS77}, this equation defines a finite \'etale cover $\tau:
U'\to U$, restricting to $\tau_C: C'\to C$ to $C$.  As $U$ is simply connected,
$\tau$ is trivial, so is $\tau_C$. 
Thus $ (\tau_C)_*\sO_{C'}$ is trivial, and so is $\iota^* E\subset
(\tau_C)_*\sO_{C'}$.
Thus by Theorem~\ref{thm:bost}, $E$ is trivial as well. 
This is a contradiction if $r\ge 2$. For $r=1$, we use in addition that the
points $[j_*E]$ constructed are dense in $\sU$, and therefore $\sU=[\sO_X]$.
Thus $E_n=\sO_U$ for infinitely many $n$, thus by Lemma~\ref{lem:tr_fncs} i), 
the $F$-divided sheaf $(E_n)_n$ is trivial. This finishes the proof.

\end{proof}
\begin{rmk} \label{rmk:mainthm}
In Theorem~\ref{thm:mainthm} one has a geometric assumption, $X$ being normal,
and an arithmetic one, $k=\bar \F_p$. \\[.1cm]
The geometric assumption is necessary to define, for each stratified bundle
$E=(E_n, \sigma_n)_{n\ge 0}$ on $U$, an extension $j_*E_n$ on $X$ for which one
can show the boundedness theorem~\ref{thm:boundedness}.  If $U$ does not admit a
normal compactification with boundary of codimension $\ge 2$, we do not know how
to bound the family of $E_n$, even if we assume that $E$ is regular singular.
The analogous question in complex geometry is interesting. We asked P. Deligne
whether over $\C$, given a smooth compactification $j: U\to X$ such that the
boundary $X\setminus U$ is a normal crossings divisor, the set of all 
Deligne canonical extensions $(E_X, \nabla_X)$ of regular singular algebraic
flat connections $(E, \nabla)$ on $U$ of bounded rank fulfills: the set
$\{c_i(E_X)\}$ in the Betti cohomology algebra  $\oplus_i H^{2i}(U,  \Z)$ is
bounded. 
The answer is yes \cite{Del14};  the proof, which is non-algebraic, uses as a
key tool,
that the topological fundamental group is finitely generated, and thus there is
an affine Betti moduli space, etc. Our aim in this article is precisely to
overcome the lack of such a finitely generated abstract group which controls
$\sO_X$-coherent $\sD_X$-modules.

\medskip

As for the arithmetic assumption, we could drop it, if we had  a
specialization homomorphism on the \'etale fundamental group with suitable
properties. 
We discuss this in Section~\ref{sec:spe}.

\end{rmk}
\section{Specialization} \label{sec:spe}
In \cite[X]{SGA1},  Grothendieck shows the existence of a continuous
specialization  homomorphism
$\pi_1^{\rm \acute{e}t}(X_{\bar K})\to \pi_1^{\rm \acute{e}t}(X_{\bar k})$ for
$f: X\to S$ a proper morphism,  with geometric connected fibers, that is
$f_*\sO_X=\sO_S$, $S$ integral, $K$ the function field of $S$ and $ \Spec k\to
S$ a closed point
(\cite[p.~207]{SGA1}) and he shows in \cite[Cor.~2.3]{SGA1} that this morphism
is surjective if $f$ is separable.  Moreover it is shown in
\cite[XIII,~Thm.~2.4]{SGA1} that the existence of the specialization
homomorphism
extends to the tame fundamental group  $\pi_1^{\rm \acute{e}t, t}(U_{\bar K})\to
\pi_1^{\rm \acute{e}t,t}(U_{\bar k})$ if $U\to S$ is the complement of a
relative normal crossings divisor in $f: X\to S$, satisfying the previous
assumptions, proper with $f_*\sO_X=\sO_S$.  This specialization homomorphism is
an isomorphism on the prime to $p$-quotient.  That the tameness is necessary is
of course visible already for $X=\mathbb{P}^1$, $U=\mathbb{A}^1$, $S=\Spec
\mathbb{Z}$ as here
$\pi_1^{\rm \acute{e}t, t}(U_{\bar K})=\pi_1^{\rm \acute{e}t}(U_{\bar K})
=0$ while $ \pi_1^{\rm \acute{e}t}(U_{\bar k})$ is as huge as Abhyankar's
conjecture predicts. 
The aim of this section is to show the existence of examples in pure
characteristic $p>0$ over a base, for which over the geometric generic fiber,
the
fundamental group has no $\Z/p$-quotient, that is the variety has no
Artin-Schreier covering, while for all geometric fibers of it over $\bar\F_p$,
it 
does have such non-trivial coverings.

\medskip
Let $k=\bar \F_p$, $ C'\subset \mathbb{P}^2$ be a smooth elliptic curve, $K$ be
$k(
C')$,  $x_i: \Spec K\to C', \ i=1,\ldots, 9$ be $9$ $K$-valued points
such that if $X\to \mathbb{P}^2_K$ is the blow up of those points, with strict
transform
$C_K$ of $C'_K$, then $\sO_{C_K}(C_K)\in \Pic^0(C_K)$ is not  torsion (this is 
arranged by choosing 8 points to be distinct $k$-points, base-changed, and the 
ninth to be a $k$-generic point). Here $X$ is defined over $K$, and so is  
$U=X\setminus C_K$.  We take a model $X_R\to \Spec R$ of $X\to \Spec K$, where 
$\Spec R$ is a non-empty open in $C'_0/\F_q$ where $C'_0/\F_q$ descends $C'/k$, 
over which the $x_i$ are rational, and are disjoint sections. So  $C \subset X$ 
over $K$ has a model $C_R\subset X_R$ over $R$ and 
$U_R:=X_R\setminus C_R \to \Spec R$ is the complement of a strict relative
normal
crossings divisor. 

\medskip

\begin{lem}
For any closed point    $a:\Spec \F_{q^m} \to \Spec R$,
\begin{itemize}
\item[i)] $\Gamma(U_a, \sO)$ has transcendence degree $ 1$ over $
\F_{q^m}$;
\item[ii)] $H^1_{\rm \acute{e}t}(U_x\otimes_{\F_{q^m}} k,\Z/p)\neq 0$;
\item[iii)] $\pi_1^{\rm \acute{e}t}(U_a\otimes_{\F_{q^m}} k)\neq 0$.
\end{itemize}
\end{lem}
\begin{proof}
Clearly i) implies ii) as the residue class of a transcendental element in
\[\Gamma(U_x\otimes_{\F_{q^m}} k, \sO)/{\rm Im}(F-1)\subset
H^1_{\rm \acute{e}t}(U\otimes_{\F_{q^m}} k, \Z/p)\] 
is non-trivial, and ii) implies iii) by
definition. 

We show i).   One has 
\[H^0(U_a, \sO)/\F_{q^m} =\varinjlim_{n\in \N}
H^0(\sO_{nC_a}(nC_a))= \varinjlim_{n \in \N} H^0(\sO_{rn C_a}(rn C_a))\] 
where $r$ is the torsion order of $\sO_{C_a}(C_a)$.  One easily computes that
${\rm
dim}_k H^0(\sO_{rnC_a}(rnC_a))= (n+1)$, and that,  as a $\F_{q^m}$-algebra,
$H^0(U_a,\sO)$ is
spanned by any element $t$ whose image spans $H^0(\sO_{rC}(rC))$. Thus
$\Gamma(U_a, \sO)=\F_{q^m}[t]$.

\end{proof}
\begin{lem}
One has
\begin{itemize}
\item[i)] $ \Gamma(U_K, \sO)=K$;
\item[ii)]  $H^1_{\rm \acute{e}t}(U_{\bar K}, \Z/p)=0$.
\end{itemize}
\end{lem}
\begin{proof}
The condition on $\sO_C(C) \in \Pic^0(C)$ not being torsion implies
$H^i(\sO_{nC}(nC))=0$ for all $n\ge 1$ and $i=0,1$.  This 
 immediately implies that i), and also ii), as  
$H^1_{\rm \acute{e}t}(U_{\bar K}, \Z/p)=\big(\varinjlim_{n\in \N}
H^1(\sO_{nC}(nC))\big)^{F=1}$.

\end{proof}
\begin{rmk} It would be desirable to understand whether or not $\pi_1^{\rm
\acute{e}t}(U_{\bar
K})=0$. This is equivalent to saying that for any finite \'etale map $h: V\to
U$,
$H^1_{\rm \acute{e}t}(V, \Z/p)=0$. 
For any finite flat map $h: V\to U$, such that $V$ is integral, one has $H^0(V,
\sO_V)=k$. Indeed, $H^0(V, \sO)$ is then a finite $k=H^0(U, \sO)$-algebra, thus
is an
artinian $k$-algebra, and is integral, thus is equal to $k$. But this is not
enough to conclude that $H^1_{\rm \acute{e}t}(V, \Z/p)=0$ and we have not
succeeded in computing it.
\end{rmk}
\section{Some theorems over any field} \label{sec:field}
\subsection{A purely geometric example}  As $\A^1$ has a huge fundamental group
in positive characteristic, 
it is not easy to find examples of smooth non-proper simply connected varieties
which are not simply the complement of a codimension $\ge 2$ subscheme in a
smooth projective simply connected variety.
We construct in this section a simple example of a normal surface for which the
maximal 
smooth open subvariety is simply connected. 

\medskip

Let 
$Y$ be a projective smooth variety  over an algebraically closed field $k$ of
characteristic $p$, and let
$L=\sO_Y(\Delta)$ be a line bundle, where $\Delta$ is a non-empty  effective
divisor.
We define  
\[ U=  \mathbb{P}(\sO_Y \oplus L)  \setminus \infty\mbox{-section},\] 
so $f:U\to Y$ is the total space of the geometric line bundle $L$. Note that if
$L$ is globally generated 
and ample, then $U$ also appears as the complement of the vertex of the
corresponding normal projective cone $X$ over 
$Y$ (which is ${\rm Proj} \oplus_{n\ge 0}H^0(Y,Sym^n (\sO_Y\oplus L))$). This
also gives concrete examples as in  
Theorem~\ref{thm:mainthm}, where we can get the conclusion in an alternate way,
because of the special geometry.  

\begin{prop}
\begin{itemize}
\item[i)]
$f_*:\pi_1^{\rm{\acute{e}t}} (U) \to \pi_1^{\rm{\acute{e}t}}(Y)$ is an
isomorphism, 
and $f^*$ induces an isomorphism on the category of stratified bundles.
\item[ii)] If $\pi_1^{\rm{\acute{e}t}}(Y)=0$, there are no non-trivial
stratified bundles on $U$.
\end{itemize}
\end{prop}
\begin{proof}
ii): we apply i) and  \cite[Thm.~1.1]{EsnMeh10}   to $Y$. We are left with i):
the section  $ \sigma_0: \sO_Y(-\Delta)\to \sO_Y$ induces ${\rm Sym} \ 
\sO_Y(-\Delta)\to {\rm Sym} \ \sO_Y$ thus a morphism $\sigma: Y\times_k \A^1\to
U$ which is
an isomorphism on $U\setminus  f^{-1}(\Delta)$,  is birational dominant.
We fix base points $(y_0\in Y$, $0\in \A^1)$, where $y_0\in {\rm
supp}(\Delta)$, 
and set  $\sigma(y_0,0)=x_0$. Thus $\sigma_*: \pi_1^{\rm{\acute{e}t}}(Y\times_k
\A^1, (y_0,0))\to
\pi_1^{\rm{\acute{e}t}}(U,x_0)$ is surjective. One has the  K\"unneth formula
$\pi_1^{\rm{\acute{e}t}}(Y\times_k
\A^1,(y_0,0))=\pi_1^{\rm{\acute{e}t}}(Y,y_0)\times
\pi_1^{\rm{\acute{e}t}}(\A^1,0)$
as $Y$ is proper  \cite[Exp.~X,~Cor.~1.9] {SGA1}. 
Furthermore, $\sigma(y_0 \times \A^1)= x_0$, so that the composition
$\{y_0\}\times\A^1 \to U$ induces the trivial homomorphism
on fundamental groups. Hence the morphism $Y\times\{0\}\to U$, induced by
restricting $\sigma$, is a surjection on fundamental groups, 
which must be an isomorphism, as seen by composing with the projection $f:U\to
Y$.  Hence $f$ induces an isomorphism on fundamental groups as well. 

On the  other hand, the K\"unneth formula  holds for the stratified fundamental
group
\cite[Prop.~2.4]{Gie75}. Thus a very similar argument shows that if $E$ is
stratified on $U$, 
$\sigma^*(E)=p_Y^*(E')$ for some $E'$ on  $Y$, $p_Y: Y\times_k \A^1\to Y$. So
$E$ and
$f^*(E')$ are isomorphic on a non-trivial open of $U$,  thus are isomorphic.
Indeed 
$\pi^{\rm strat}(U)\to \pi^{\rm strat}(X)$ is surjective 
(\cite[Comment~before~Thm.3.5]{Phu12} and \cite[Lem.~2.5]{Kin14}).
Again restricting to the $0$-section shows $E'$ is unique. 

\end{proof}

\subsection{Second main theorem}
The aim of this section is to 
give a general statement which does not require $k$ to be isomorphic to the
algebraic 
closure of a finite field.  The proof is a simple application of
Theorem~\ref{thm:bost}. 
It is nonethelss worth being mentioned as it yields new concrete examples, see
\ref{ex}.
\begin{thm} \label{thm2}
Let $X$ be a normal projective variety of dimension $d\ge 1$,  defined over an
algebraically closed field $k$.  Let $U$ be its regular locus. If  $\pi_1^{\rm
\acute{e}t}(U)=0$, and $U$ contains a projective curve $C$, intersection of
$(d-1)$-hyperplane sections, such that $\pi_1^{{\rm \acute{e}t}}(C)$ is abelian,
then there are no non-trivial $\sO_X$-coherent $\sD_X$-module on $U$.

\end{thm}
\begin{prop} \label{prop:ext}
Let $C\to U$ be a proper subscheme of a smooth quasi-projective variety 
defined over an algebraically closed field $k$ 
 of characteristic $p>0$.  If $\pi^{\rm{strat}}(C)\to  \pi^{\rm strat} (U)$ is
surjective, then any  $\sO_X$-coherent $\sD_X$-module on $U$ which is an
extension of
$\mathbb{I}$ by itself has finite monodromy.  

\end{prop}

\begin{proof}
As  ${\rm
Ext}^1_C(\mathbb{I}, \mathbb{I})=H^1_{\rm \acute{e}t}(C, \Z/p)\otimes_{\F_p} k$
(\cite[(9)]{dSan07}),  
restriction to $C$ of any  $\sO_X$-coherent $\sD_X$-module on $U$ which is an
extension of
$\mathbb{I}$  by itself has finite monodromy, thus by the surjectivity
$\pi^{\rm{strat}}(C)\to  \pi^{\rm strat} (U)$, 
 any  $\sO_X$-coherent $\sD_X$-module on $U$ which is an extension of
$\mathbb{I}$ by itself has finite monodromy as well. 

\end{proof}

\begin{proof}[Proof of Theorem~\ref{thm2}]
By  Proposition~\ref{prop:Fdiv} ii), we may assume that $C$ is semi-normal. 
This implies  that
\begin{itemize}
\item[i)] 
 either all components of $C$ are smooth rational curves and the homology 
of the dual graph of $C$ is $0$ (so the graph is a tree) or $1$-dimensional, 
\item[ii)]  or else $C$ is the union $C'\cup C''$, where 
$C'$ is a tree of smooth rational curves, and $C''$ is irreducible, so is 
either a smooth elliptic 
curve 
or an irreducible rational with one node (called 1-nodal curve in the sequel), 
and $C'\cap C''$ consists of one point.  
\end{itemize}
For any smooth irreducible component $C_0$ of $C$, $\pi^{\rm strat}(C_0)$ is
trivial if
$C_0$ is rational, and  abelian if $C_0$ is an elliptic curve
(\cite[p.~71]{Gie75}). 
If the graph is a tree, then the bundles $E_n$ of a $F$-divided sheaf $E=(E_n,
\sigma_n)_{n\in \N}$
on $C$ are uniquely determined by their restriction to the components $E_i$, so 
 $\pi^{\rm strat}(C)$ is $0$ if all components are rational, else (case (ii)
above) is equal to
$\pi^{\rm strat}(C'')$.

Assume $C$ is a graph of rational curves with $1$-dimensional homology. One
shows easily that
any indecomposable vector bundle of rank $r$ on $C$, whose pull-back to the
normalization is 
trivial, must be isomorphic to $L\otimes E_r$, where $L$ is a line bundle, and
$E_r$ is the 
unique indecomposable vector bundle of rank $r$ which is unipotent (multiple
extension of $\sO_C$). 
Indeed, we may write $C=C_1\cup C_2$ where $C_1$ is a tree of smooth rational
curves, $C_2$ an 
irreducible rational curve, and $C_1\cap C_2$ consists of two points. Then
trivializing the 
restriction of the bundle to 
$C_1$  and $C_2$ in a compatible manner at one of the intersection points of
$C_1$ and $C_2$, 
 the bundle 
is determined up to isomorphism by the conjugacy class of the element of
$GL_r(k)$ given by the 
glueing data at the second point of $C_1\cap C_2$. Using the Jordan canonical
form we deduce the 
description of bundles.

A similar description holds when $C=C'\cup C''$, where $C'$ is a tree of smooth
rational 
curves, and $C''$ is 1-nodal. In both of these cases, we note the resemblance to
the Atiyah 
classification of indecomposable bundles on elliptic curves. 

The proof in \cite[p.~71]{Gie75} for elliptic curves, that $\pi^{\rm{strat}}$
is abelian, (since any irreducible $\sO_X$-coherent $\sD_X$-module is of rank 1,
which results  from the 
explicit description using Atiyah's classification), adapts to the other two
remaining cases. It shows 
that  in the situation where the homology of the graph is 1-dimensional, and
where one component is 1-nodal, 
$\pi^{\rm strat}$ is again abelian. There is a simplification resulting from the
fact that ${\rm Pic}^0$ has 
no $p$-torsion, so that line bundles which admit an$\sO_X$-coherent
$\sD_X$-module admit only one. 

By Theorem~\ref{thm:bost}, this implies  that $\pi^{\rm strat}(U)$ is abelian,
so
all irreducible  objects have rank $1$. Since $\pi_1^{\rm{\acute{e}t}}(U)=0$
implies in particular that $\Pic(U)$ is finitely generated, the rank $1$ objects
are trivial, so all we have to understand are extensions of $\mathbb{I}$ by
itself when $\pi_1^{\rm{\acute{e}t}}(C)$ is not-trivial.  We apply
Proposition~\ref{prop:ext} to conclude. 
\end{proof}
 \begin{ex} \label{ex}
Assume ${\rm char.}\,k\neq 2,3$. A non-trivial example of Theorem~\ref{thm2} is
as follows. Let $X\subset {\mathbb P}^3$ be given by an equation 
\[f(x,y,z)w+g(x,y,z)=0\] 
where $f=0$ defines a smooth plane cubic, and $g=0$ is a plane quartic which (i)
intersects the plane cubic $f=0$ transversally, 
and (ii) has abelian \'etale fundamental group, so that it is singular. Thus the
semi-normalization  of $(g=0)$ is as described in the proof of 
 Theorem~\ref{thm2}. 
 
It is easy to see that $X$ has a unique singular point, given by
$P=\{x=y=z=0,w=1\}$, whose complement $U$ contains the ample curve given by
$w=0$, 
which is the plane curve $g=0$.  One also sees that the blow up
$\pi:\tilde{X}\to X$ of $X$ at its singular point is identified with the blow up
of 
${\mathbb P}^2$ at the 12 points given by $f=g=0$, and the exceptional divisor
$\pi^{-1}(P)$ is identified with the strict transform of the plane 
cubic $f=0$ under the blow-up of ${\mathbb P}^2$.
 
The surface $U$ satisfies $\pi_1^{\rm{\acute{e}t}}(U)=0$, because (i) the
existence of the ample curve $C$ with abelian fundamental group implies that 
the fundamental group of $U$ is also abelian (ii) using that $U=X\setminus\{P\}$
is identifed with an open subset of the blow-up of ${\mathbb P}^2$ 
obtained by removing the strict transform of the cubic $f=0$, we see that the
tame fundamental group of $U$ is trivial (iii) a direct computuation 
now shows also that $H^1_{\rm{\acute{e}t}}(U,{\mathbb Z}/p{\mathbb Z})=0$ by
considering the Frobenius action on $H^1(U,\sO_U)$. 

 Here are some explicit examples: assume ${\rm char}. k\neq 2,3$; take
$f=x^3+y^3+z^3$. The following choices of $g$ satisfy the conditions above:\\
 (i) $g=x^4+ay^3z$, where $a=1$ or $-1$ depending on $p$ (here $g=0$ is simply
connected)\\
 (ii) $g=(x^2+ay^2-4z^2)(x^2+by^2-4z^2)$, where $a\neq b$ are suitably chosen,
depending on $p$ (here $g=0$ consists of 2 conics tangent at 2 points, 
 so has abelian fundamental group); a variant is obtained by taking $b=0$, so
that the quartic $g=0$ is the union of a conic and two tangent lines.
 \end{ex}
\begin{rmk} In the above example, $X$ is a normal quartic surface with a unique
singular point,  whose complement $U$ is simply connected, 
when $U$ contains a ``special'' singular hyperplane section. It seems plausible
to guess that $U$ is simply connected for any such quartic 
with a triple point, even if no such special hyperplane section exists. 
\end{rmk}
\bibliographystyle{plain}
\renewcommand\refname{References}

\appendix \newpage


\section{ Algebraization of formal morphisms and  formal subbundles, 
{\tiny by} {\small Jean-Beno\^it Bost} } \label{app}

In this Appendix, we establish the algebraicity criterion for vector bundles on
which the proof of Theorem 3.5 relies.

Our derivation of this criterion relies on two classical results: an
\emph{algebraicity criterion} for formal subschemes, which already appears in
\cite{Bost01} in a quasi-projective setting, and actually is a straightforward
consequence of the study by Hartshorne (\cite{Hartshorne68}) of the fields of
meromorphic functions on some formal schemes, and a \emph{connectedness theorem}
\emph{à la} Bertini--Fulton--Hansen (\cite{FultonHansen79}). 

In the first two  sections of this Appendix, we recall these two results in a
form suitable for our purposes. Then we prove an algebraicity criterion
concerning formal morphisms, from which we finally deduce algebraicity results
for formal vector bundles.

Let us also mention that the type of algebraicity results discussed in this
Appendix originates in Grothendieck's seminar SGA2 (\cite{GrothendieckSGA2}) and
in the work of Hironaka and Matsumura (\cite{Hironaka68},
\cite{HironakaMatsumura68}).

We denote by $k$ an algebraically closed field. All $k$-schemes are assumed to
be separated, of finite type over $k$.

In this Appendix, we could have worked with quasi-projective $k$-schemes only.
Our results still hold in this quasi-projective setting\footnote{This is formal,
but for the implication $(1) \Rightarrow (2)$ in Proposition \ref{algequ};
however, the validity of the quasi-projective version of this implication
directly follows from its proof below.}, which would be enough for the proof of
Theorem 3.5. Actually several of our arguments become more elementary in this
setting; this is for instance the case of the proofs of Theorem
\ref{thalgnrample} and  Proposition \ref{etZar}. 

The exposition of the results in this Appendix has benefited from the
suggestions of an anonymous referee, to whom I am grateful.

\subsection{Algebraization of formal subschemes}

 Let $Z$ be a $k$-scheme over $k$, and $Y$ a closed  subscheme of $Z$. We shall
denote by $i: Y \hra Z$ the inclusion morphism, and by $\hZ_Y$ the formal
completion of $Z$ along $Y$.

 Let $\hV$ be a closed $k$-formal subscheme of $\hZ_Y$, which admits $Y$ as
scheme of definition\footnote{In other words, it contains $Y$ and its underlying
topological space $\vert \hV \vert$ coincides with $\vert \hZ_Y \vert = \vert Y
\vert.$} and which is formally \emph{smooth} over $k$ (or equivalently is
regular), of pure dimension $d$. 
The \emph{Zariski closure} $\overline{\hV}$ is the smallest closed subscheme of
$Z$ which, considered as a formal scheme, contains $\hV$, or equivalently, the
smallest closed subscheme of $Z$  which contains $Y$ and whose formal completion
along $Y$ contains $\hV$. The smoothness of $\hV$ implies that $\overline{\hV}$
is reduced, and is irreducible if $Y$ is connected. The  dimension of 
$\overline{\hV}$ is at least $d$. 

\begin{proposition}\label{algequ} With the above notation, the following
conditions are equivalent:
\begin{enumerate}
\item The Zariski closure $\overline{\hV}$ of $\hV$ in $Z$ has dimension $d$.
\item There exists a smooth $k$-scheme $M$ and  morphisms of $k$-schemes $j:Y
\lra M$ and $\mu: M \lra Z$ such the diagram
\begin{equation}\label{triang}
\xymatrix{ & M \ar[d]^\mu\\ 
Y\ar[r]^i \ar[ur]^j & Z}
\end{equation}
is commutative 
and the associated morphism of formal completion 
$\hat{\mu}: \hM_{j(Y)} \lra \hZ_Y$ maps $\hM_{j(Y)}$ isomorphically onto $\hV$:
\begin{equation}\label{muiso}
\hat{\mu}: \hM_{j(Y)} \lrasim \hV.
\end{equation}
\end{enumerate}
\end{proposition}

(Observe that the commutativity of (\ref{triang}) implies that $j$ is a closed
immersion.)


When the equivalent conditions in Proposition \ref{algequ} conditions are
satisfied, we shall say that \emph{the smooth formal scheme $\hV$ is algebraic.}

\begin{proof} To prove the Proposition, we may assume that $Y$ is connected.

Suppose that (1) is satisfied, and consider the normalization
$$\nu: \overline{\hV}^{\rm nor} \lra  \overline{\hV}$$
of the integral scheme $\overline{\hV}$. We may consider the completion 
$\overline{\hV}^{\rm nor \wedge}_{\nu^{-1}(Y)}$ of $\overline{\hV}^{\rm nor}$
along $\nu^{-1}(Y)$. By the ``analytic normality of normal rings", it is a
normal formal scheme
\footnote{Namely, 
for any affine connected open  subscheme $U$ of $\vert \overline{\hV}^{\rm nor
\wedge}_{\nu^{-1}(Y)} \vert = \vert \nu^{-1}(Y) \vert$, the ring $\hat{A}$ of
sections over $U$ of the structure sheaf  of $\overline{\hV}^{\rm nor
\wedge}_{\nu^{-1}(Y)}$ is a normal domain. Actually, if $V$ denotes a connected
affine open subscheme of $\overline{\hV}^{\rm nor}$ such that $U= V \cap \vert
\nu^{-1}(Y) \vert,$ the ring $\hat{A}$ may be identified with the completion of
$A:= \Gamma(V, \cO_{\overline{\hV}^{\rm nor}})$ in the $I$-adic topology, where
$I$ denote the ideal of $A$ defining the closed subscheme $U$ of $V$. 
The fact that this completion is normal easily follows from the fact that, for
every closed point $x$ of $V$, the completion $\hat{A}_x$ of $A$ at $x$ is
itself a normal domain: this last fact is precisely the content of the
``analytic normality" of the normal ring $A$, as established for instance in
\cite{ZariskiSamuelII}, p. 313-320, or \cite{Nagata62}, Section 37.5. See for
instance \cite{Badescu04} p. 102 for more details.
}
 and there is a unique morphism of formal scheme 
$$\hat{j}: \hV \lra  \overline{\hV}^{\rm nor \wedge}_{\nu^{-1}(Y)}$$
such that the following diagram is commutative:
\begin{equation}\label{triangbis}
\xymatrix{ &  \overline{\hV}^{\rm nor \wedge}_{\nu^{-1}(Y)}
\ar[d]^{\hat{\nu}}\\ 
\hV \ar[r]^{\iota} \ar[ur]^{\hat{j}} & {\overline{\hV}}^{\wedge}_{Y}, }
\end{equation}
where $\iota: \hV \ra {\overline{\hV}}^{\wedge}_{Y}$ denotes the inclusion
morphism. The morphism $\hat{j}$ is a closed immersion, and at the level of
local rings, defines surjections of integral domains of the same dimension,
which are therefore isomorphisms.  (See \cite{Gieseker77}, Proof of  Theorem
4.3, p. 1148-1149, and \cite{Badescu04}, Proof of Theorem 9.16, p. 100-103, for
similar arguments.)

 In other words, if we let 
$$j:= \hat{j}_{\vert Y},$$
then $j$ defines a closed immersion of $Y$ in $\overline{\hV}^{\rm nor}$, its
image is a connected component of $\nu^{-1}(Y)$, and $\hat{j}$ establishes an
isomorphism from $\hV$ to the formal completion   $\overline{\hV}^{\rm nor
\wedge}_{j(Y)}$. 

This shows that $\overline{\hV}^{\rm nor}$ is smooth on some open neighborhood
of $j(Y)$, and Condition (2) is satisfied by the regular locus
$M:=\left(\overline{\hV}^{\rm nor}\right)_{\rm reg}$ of $\overline{\hV}^{\rm
nor}$ and 
$\mu := \nu_{\vert M}.$

Conversely, when Condition (2) holds, we may assume that $M$ is connected, after
possibly replacing it by its connected component containing $j(Y)$. The
isomorphism (\ref{muiso}) shows that $M$, like $\hV$, has dimension $d$.
Consequently $\overline{\mu(M)}$ has dimension at most $d$, and its subscheme
$\overline{\hV}$ also.
\end{proof}

\begin{theorem}\label{thalgnrample} With the above notation, if $Y$ is proper
over $k$ and a l.c.i. in $\hV$, with all components of dimension $\geq 1,$ and
if its normal bundle $N_Y \hV$ in $\hV$ is ample, then $\hV$ is algebraic.
\end{theorem}

This theorem may be seen as an avatar, valid in algebraic geometry (possibly
over a field of positive characteristic) of the classical theorem of Andreotti
about the algebraicity of pseudo-concave complex analytic submanifolds of the
complex projective spaces (\cite{Andreotti63}, Th\'eor\`eme 6). 

The proof below relies on some upper bound, established by Harshorne
(\cite{Hartshorne68}), on the transcendence degree of the field of meromorphic
functions --- the ``formal-rational functions" investigated by  Hironaka and
Matsumura (\cite{Hironaka68}, \cite{HironakaMatsumura68}) --- on the formal
scheme $\hV$.  We refer the reader to \cite{Gieseker77}, Section 4,  for
geometric applications of Hartshorne's results in a similar spirit.
 
In a quasi-projective setting, Theorem \ref{thalgnrample} admits a simpler
proof, inspired by techniques of Diophantine approximation, which avoids the use
of formal meromorphic functions (\cite{Bost01}, Section 3.3). Actually, as shown
in  \cite{Chen12}, these techniques  lead to algebraization results valid under
assumptions significantly weaker that the ampleness of  $N_Y \hV$ when $\dim Y >
1.$

\begin{proof} We may assume that $\hV$ (or equivalently $Y$) is connected. Then
we may consider the field $k(\hV)$ of meromorphic functions on $\hV$, as
introduced in \cite{Hironaka68}, p. 600, and \cite{HironakaMatsumura68}, §1.
Recall that, as a ring, it is defined as  the space of global sections
$\Gamma(\hV, \mathcal{M}_{\hV})$ of the sheaf $\mathcal{M}_{\hV}$ on $\hV$
defined as the associated sheaf of the presheaf $\mathcal{M}^0_{\hV}$ defined
by 
sending an arbitrary open subset $U$ of $\hV$ (or equivalently of $Y$) to the
total ring of fractions $\mathcal{M}^0_{\hV}(U)$ of $\cO_{\hV}(U).$ In the
present situation where $\hV$ is assumed connected, $k(\hV)$ is actually a field
(for instance,  by a simple variant of the discussion in \cite{Hartshorne70},
pp. 189--190; see also \cite{Badescu04}, Chapter 9, for an exposition of the
basic properties of formal rational functions).

According to \cite{Hartshorne68}, Theorem 6.7, the transcendence degree
$\degtr_k k(\hV)$ of the field $k(\hV)$ is at most $d$.

Besides, the $k$-morphism 
$\hV \hra \overline{\hV}$ induces a morphism of extensions of $k$:
$$k(\overline{\hV}) \lra k(\hV).$$
Indeed, by the very definition of $\overline{\hV}$, for any open subscheme $U$
of $\overline{\hV}$ which meets $\vert Y \vert,$ the restriction morphism
$$\cO_{\overline{\hV}}(U) \lra \cO_{\hV}(U \cap \vert Y\vert)$$
is injective.

This implies that
$$\dim \overline{\hV} = \degtr_k k(\overline{\hV}) \leq \degtr_k k(\hV) \leq
d,$$
and establishes that $\hV$ is algebraic.
\end{proof}

\subsection{Connectedness and \'etale neighborhoods of intersections of ample
hypersurfaces} 

\subsubsection{A connectedness theorem}

The following theorem is an avatar of the classical connectedness theorem of
Fulton-Hansen \cite{FultonHansen79}. Actually, its special case   when the 
inclusion morphism $H_1 \cap \ldots \cap H_e \hra X$ is a regular imbedding of
codimension $e$, 
is alluded to by Fulton and Hansen in \cite{FultonHansen79}, p. 161. The general
case is essentially established by Debarre in \cite{Debarre96}.

\begin{theorem}\label{new} Let $X$ be an integral projective $k$-scheme, $e$ a
positive integer, and $H_1, \dots, H_e$ ample effective Cartier divisors on $X$.

For any integral proper $k$-scheme $X'$ and any $k$-morphism $f: X' \lra X$ such
that 
$$\dim f(X') > e,$$
the inverse image $f^{-1}(H_1 \cap \ldots \cap H_e)$ is a non-empty
\emph{connected} subscheme of $X'$.  
\end{theorem}

For every $i \in \{1, \dots, e\}$, we may choose a positive integer $D_i$ such
that the Cartier divisor $D_i H_i$ is very ample. In other terms, there exists a
projective embedding $\iota_i : X \lra \PP_k^{N_i}$ and an hyperplane $L_i$ in
$\PP_k^{N_i}$, which does not contain $\iota_i(X)$, such that the following
equality of divisors hold:
$$D_i H_i = \iota_i^{-1}(L_i).$$

Then, if we let 
$$L:= \prod_{1 \leq i \leq e} L_i,$$
we have
\begin{equation*} 
\begin{split}
 \vert H_1 \cap \dots \cap H_e \vert & = \vert D_1 H_1 \cap \dots \cap D_e H_e
\vert \\
 & = \vert \iota_1^{-1} (L_1) \cap \dots \cap \iota^{-1}_e(L_e) \vert \\
 & = \vert (\iota_1, \dots, \iota_e)^{-1}(L)\vert.
\end{split}
\end{equation*}  
Consequently the support of $f^{-1}(H_1 \cap \ldots \cap H_e)$ coincides with
the one of the inverse image of $L$ by the composite morphism
$$X'\stackrel{f}{\lra} X \stackrel{(\iota_1,\dots, \iota_e)}{\lra}\prod_{1\leq i
\leq e} \PP_k^{N_i}.$$

 Theorem \ref{new} is easily seen to follow from \cite{Debarre96}, Th\'eor\`eme
1.4, 2) a), applied to this morphism and to the subspace $L$ of $\PP:=
\prod_{1\leq i \leq e} \PP_k^{N_i}.$

For the sake of completeness, we sketch a proof of Theorem \ref{new} from
``basic results" in algebraic geometry, at the level of Harsthorne's textbook
\cite{Hartshorne83} and Jouanolou's monograph \cite{Jouanolou83}. We also refer
the reader to  \cite{FultonLazarsfeld81b} for a beautiful discussion of related
connectedness results. 
(Actually, the proof of Th\'eor\`eme 1.4 in \cite{Debarre96} relies  on results
and techniques presented in   \cite{FultonLazarsfeld81b}.) 

\begin{proof} By considering the Stein factorization of $f$ (see \cite{EGAIII1},
Section III.4.3), we may assume that $f$ is a \emph{finite} morphism. Then, for
every $i \in \{1,\ldots,e\},$ the following alternative holds: either $f(X')
\subset H_i,$ or $f^{\ast}(H_i)$ defines an ample effective divisor in $X'$.
Moreover, if $I$ denotes the set of elements $i$ in $\{1, \dots, e\}$ such that
the second case arises, we clearly have:
$$f^{-1}(H_1 \cap \ldots \cap H_e) = \bigcap_{i \in I} f^{-1}(H_i).$$  

These observations show that, to establish Theorem \ref{new} in full generality,
it is enough to establish the following assertion, which is actually the special
case of Theorem \ref{new} when $X'=X$ and $f=Id_X$:

\emph{ Let $X$ be an integral projective $k$-scheme, $e$ a positive integer, and
$H_1, \dots, H_e$ ample effective Cartier divisors on $X$. If $\dim X > e$, then
$H_1 \cap \ldots \cap H_e$ is a connected subscheme of $X$.}

To establish this assertion, let us choose, for every $i \in \{1, \ldots, e\}$,
a positive integer $D_i$ such that the line bundle $\cO_X(D_i H_i)$ is very
ample, and let us consider the associated projective embedding\footnote{We use
Grothendieck's notation: the projective space $\PP(\Gamma(X, \cO_X(D_iH_i))$
parametrizes hyperplanes in  $\Gamma(X, \cO_X(D_iH_i))$.}
$$f_i : X \lra \PP(\Gamma(X, \cO_X(D_iH_i)) \simeq \PP_k^{N_i}$$
(where $N_i := \dim_k \Gamma(X, \cO_X(D_iH_i) - 1$) and the product imbedding
$$f:= (f_1, \ldots, f_e) : X \lra \PP := \prod_{1\leq i \leq e} \PP(\Gamma(X,
\cO_X(D_iH_i)) \simeq \prod_{1\leq i \leq e} \PP_k^{N_i}.$$

As usual, we denote $\vert D_i H_i \vert$ the projective space dual to
$\PP(\Gamma(X, \cO_X(D_iH_i))$. We also consider the incidence correspondences
 \begin{equation*}
\xymatrix{ & I_i \ar[dl]_{p_i} \ar[dr]^{q_i} &\\ 
\PP_k^{N_i} & &  \vert D_i H_i \vert}
\end{equation*}
and their product
\begin{equation}\label{prodincid}
\xymatrix{ & \prod_{i=1}^e I_i \ar[dl]_{p} \ar[dr]^{q} &\\ 
\PP & &   \prod_{i=1}^e\vert D_i H_i \vert.}
\end{equation}

By ``restriction" to $X$  of the correspondence (\ref{prodincid}) by means of
the imbedding $f: X \hlra \PP,$ we get a correspondence: 
\begin{equation*}
\xymatrix{ & Z:= X \times_{\PP} \prod_{i=1}^e I_i  \ar[dl]_{p_Z} \ar[dr]^{q_Z}
&\\ 
X & &   \prod_{i=1}^e\vert D_i H_i \vert.}
\end{equation*}

The morphism $p$, and consequently $p_Z$, is a ``Zariski locally trivial
fibration", the fibers of which are  isomorphic to $\prod_{i=1}^e \PP_k^{N_i
-1}.$ Therefore $Z$ is an integral projective scheme of dimension
$$\dim Z = \dim X + \sum_{i=1}^e N_i - e.$$
Moreover the morphism $q_Z$ is surjective; indeed $\dim X \geq e$, and therefore
any $e$-tuple of ample divisors on $X$ has a non-empty intersection.  Since 
$$\dim   \prod_{1\leq i \leq e}\vert D_i H_i \vert = \sum_{i=1}^e N_i,$$
the generic fiber of $q_Z$ has dimension $\dim X - e.$ 

 Observe that, if $\xi_i^0$ denotes the point of $\vert D_i H_i\vert (k)$
defined by the Cartier divisor $D_i H_i$, the intersection $D_1 H_1 \cap \ldots
\cap D_e H_e$  coincides, as a scheme, with the fiber $q_Z^{-1}(\xi_1^0, \ldots,
\xi_e^0).$ The connectedness of $H_1 \cap \ldots \cap H_e$ will therefore follow
from the connectedness of the fibers of $q_Z$, that we shall establish by means
of the following classical result, basically due to Zariski (see for instance 
\cite{Jouanolou83}, Part I, Section 4 and proof of Théorème 7.1; recall that, in
this Appendix, the base field $k$ is algebraically closed):
 
\begin{lemma}\label{ZariskiJouanolou}  Let $\phi : X_1 \lra X_2$ be a dominant
morphism of integral $k$-schemes.

a) The following conditions are equivalent:
\begin{enumerate}
\item the generic fiber of $\phi$ is geometrically irreducible;
\item there exists a dense open subscheme $U$ of $X_2$ such that, for any $P$ in
$U(k)$, the fiber $\phi^{-1}(P)$ is irreducible;
\item there exists a Zariski dense subset $\mathcal{D}$ of $X_2(k)$ such that,
for any $P$ in $\mathcal{D}$, the fiber $\phi^{-1}(P)$ is irreducible.
\end{enumerate}

b) Assume moreover that $\phi$ is proper and surjective and that $X_2$ is
normal. Then, when the conditions (1)--(3) above are satisfied, the fiber
$\phi^{-1}(P)$ is connected for every $P$ in $X_2(k)$. 
\end{lemma}

According to Lemma \ref{ZariskiJouanolou}, to complete the proof of the
connectedness of the fibers of $q_Z$, it is enough to show the existence of a
Zariski dense set of points $(\xi_1, \dots,\xi_e)$ in $\prod_{1\leq i\leq e}
\vert D_i H_i\vert (k)$ such that the fibers $q_Z^{-1}(\xi_1, \dots,\xi_e)$ are
irreducible, or equivalently, such that the schemes 
$$X \cap f_1^{-1}(\Xi_1) \cap \ldots \cap f_e^{-1}(\Xi_e)$$
--- where $\Xi_i$ denotes the hyperplane in $\PP(\Gamma(X, \cO_X(D_iH_i))$
parametrized by $\xi_i$ --- are irreducible. 

This follows by applying $e$-times  the usual Theorem of Bertini (see for
instance \cite{Jouanolou83}, Part I, Théorème 6.3, 4), Théorème 6.10, 3) and
Corollaire 6.11, 3)) to construct successively $\xi_1,$ \ldots, $\xi_e$ such
that, for every $i \in \{1, \ldots, e\},$ the intersection scheme
$$X_i := X\cap f_1^{-1}(\Xi_1) \cap \ldots \cap f_i^{-1}(\Xi_i)$$
is geometrically irreducible of dimension $\dim X - i.$  (Observe that, if
$i\leq e-1$, then $\dim X - i \geq 2$. This allows us to apply the usual Theorem
of  Bertini to $X_i$ projectively imbedded by $f_{i+1}$, and, the points
$\xi_1,\ldots,\xi_i$ being already constructed, to find a Zariski dense set of
points $\xi_{i+1}$ in   $\vert D_{i+1} H_{i+1}\vert (k)$ such that $X_i \cap
f_{i+1}^{-1}(\Xi_{i+1})$ is irreducible.)

 \end{proof}

\subsubsection{Application to \'etale neighborhoods of intersections of ample
hypersurfaces} Let $Y$ be closed subscheme of some $k$-scheme $X$, and denote
$i:Y \hra X$ the inclusion morphism. By definition, an \emph{\'etale
neighborhood} of $Y$ in $X$ is a commutative diagram of $k$-schemes:
\begin{equation}\label{etnbd}
\xymatrix{ & \tilde{X}\ar[d]^\nu \\
Y \ar[ur]^j\ar[r]^i & X,}
\end{equation}
with $\nu$ \'etale at every point of $j(Y)$. 

Observe that the commutativity of the diagram (\ref{etnbd}) implies that $j$ is 
a closed immersion. Moreover $\nu$ is \'etale at every point of $j(Y)$ if and
only if the morphism of formal completions induced by $\nu$,
\begin{equation}\label{nufor}
\hat{\nu}: \widehat{\tilde{X}}_{j(Y)} \lra\widehat{X}_Y, 
\end{equation}    is an isomorphism of formal schemes. (See for instance
\cite{Gieseker77}, Section 4.) 

Besides, after replacing $\tilde{X}$ by some open Zariski neighborhood of
$j(Y)$, we may assume that $\nu$ is an \'etale morphism. If moreover $X$ and
$\tilde{X}$ are integral, 
furthermore $X$ and $\tilde{X}$ are connected (hence integral), then 
then $\nu$ is birational iff it is an open immersion (\cf \cite{EGAIV4}, §18,
Lemme 18.10.18, p. 166).

Observe also that \'etale neighborhoods in $X$ of $Y$ and of the underlying
reduced scheme $\vert Y \vert$ may be identified. 

\begin{proposition}\label{etZar} Let $X$ be an integral proper $k$-scheme of
dimension $d$, and let $H_1,\ldots,$ $H_e$, $1\leq e \leq d-1$, be ample
effective Cartier divisors in $X$. Let us assume that $X$ is normal at every
point of $H_1 \cap \cdots \cap H_e$.

Then, for any \'etale neighborhood of $H_1 \cap \cdots \cap H_e$ in $X$
\begin{equation}\label{etnbdint}
\xymatrix{ & \tilde{X}\ar[d]^\nu \\
H_1 \cap \cdots \cap H_e \ar[ur]^j\ar[r]^{\;\;\;\;\;\;\;\;\;i} & X,}
\end{equation}
there exists a Zariski open neighborhoods $U$ of $H_1 \cap \cdots \cap H_e$ in
$X$ (resp. $\tilde{U}$ of 
of $j(H_1 \cap \cdots \cap H_e)$ in $\tilde{X}$)
between which $\nu$ establishes an isomorphism:
$$\nu_{\tilde{U}}: \tilde{U} \lrasim U.$$
\end{proposition}

In brief, Proposition \ref{etZar} asserts that an \'etale neighborhood of the
intersection of at most $d-1$ ample hypersurfaces in a normal projective variety
of dimension $d$ ``is" actually a Zariski neighborhood. 

When $d=2$ and $e=1$,  Proposition \ref{etZar} is essentially Proposition 2.2 in
\cite{BostChambert-Loir07}.  In that case, the connectedness result on which its
proof relies (namely, the special case of Theorem \ref{new} with $f$ a dominant
morphism between two projective surfaces $X'$ and $X$ and $e=1$) directly
follows from Hodge Index Theorem, by   Ramanujam's argument which shows that an
effective, nef and big divisor on a surface is numerically connected
(\cite{Ramanujam72}). 

\begin{proof}  The normality assumption on $X$ and the existence of the
isomorphism (\ref{nufor}) imply that  $\tilde{X}$ is normal on some Zariski
neighborhood of $j(Y)$. Therefore, after possibly shrinking $\tilde{X}$, we may
assume that $\tilde{X}$ is a normal scheme. According to Theorem  \ref{new}, the
scheme  $$ Y:=H_1 \cap \cdots \cap H_e$$ is connected. Consequently, its image
$j(Y)$ lies in a unique component of the normal scheme $\tilde{X}$. We may
therefore assume that $\tilde{X}$ is integral. 

By Nagata's compactification theorem, we may assume that $\tilde{X}$ is an open
subscheme of some integral proper $k$-scheme $\overline{\tilde{X}}$. After
replacing  $\overline{\tilde{X}}$ by the closure in $\overline{\tilde{X}} \times
X$ of the graph of $\nu$, we may also assume that $\nu$ extends to a morphism
$\overline{\nu}:  \overline{\tilde{X}}\ra X$.

So we may --- and will --- assume that, in (\ref{etnbd}), the scheme $\tilde{X}$
is \emph{integral and proper} over $k$.

The intersection $H_1 \cap \cdots \cap H_e$ is not empty, and $\nu$ is therefore
\'etale at some point of $\tilde{X}$. The morphism $\nu$ is therefore dominant
and its image has dimension $d > e.$ Therefore, according to Theorem \ref{new}, 
the closed subscheme
$\nu^\ast (Y) = \nu^\ast(H_1 \cap \cdots \cap H_e)$
of $\tilde{X}$ is connected.

The commutativity of the diagram (\ref{etnbdint})  shows that the subscheme
$\nu^\ast(Y)$ of $\tilde{X}$ contains $j(Y)$. Moreover, since $\nu$ defines an
isomorphism $$\hat{\nu}: \widehat{\tilde{X}}_{j(Y)} \lrasim\widehat{X}_Y,$$ the
formal subschemes of $\widehat{\tilde{X}}_{j(Y)}$ defined by completing
$\nu^\ast(Y)$ along $j(Y)$ is nothing but $j(Y)$. 

This implies that the trace of $\nu^\ast (Y)$ on some Zariski open neighborhood
of $j(Y)$ in $\tilde{X}$ coincides with $j(Y)$. (This follows form the basic
faithfulness properties of the functor of completion along a closed subscheme in
a some noetherian scheme; see for instance \cite{EGAI}, Propositions 10.8.8 and
10.8.11 and their corollaries.)
 In other words, $\nu^\ast (Y)$ may be written as a disjoint union
$$\nu^\ast (Y) = j(Y) \amalg R,$$
for some closed subscheme $R$ of $\tilde{X}$.

Together with the connectedness of  $\nu^\ast (Y)$, this shows that 
$$\nu^\ast (Y) = j(Y).$$

Consider the set $F$ of points of $\tilde{X}$ where $\nu$ is not \'etale. It is
closed in $\tilde{X}$, and disjoint of 
$$\nu^{-1} (Y) = \vert j(Y)\vert.$$
Therefore 
$$U:= X \setminus \nu(F)$$
is an open Zariski neighborhood of $Y$ in $X$, and the restriction morphism
$$\nu_{\vert \nu^{-1}(U)}: \nu^{-1}(U) \lra U$$
is proper, \'etale, and an isomorphism over the non-empty subscheme $Y$, hence
an isomorphism.
\end{proof}

\subsection{Algebraization of formal morphisms}

The following theorem provides criteria for a \emph{formal} morphism, defined on
the completion $\hX_Y$ along a closed subscheme $Y$ of a proper $k$-scheme $X$
with range some $k$-scheme,  to be defined by some morphism of $k$-schemes
defined on some \'etale or Zariski neighborhood of $Y$ in $X$.

\begin{theorem}\label{algmor} Let $X$ be an integral projective $k$-scheme of
dimension $d$, $Y$ a closed subscheme of $X$, $T$ a  $k$-scheme, and $\psi:
\hX_Y \ra T$ a morphism of $k$-formal schemes\footnote{Recall that $k$-schemes 
may be identified to $k$-formal schemes, which actually are noetherian, and
separated over $k$. In particular,  we may consider a morphism from a $k$-formal
scheme, defined for instance as a completion, to  some $k$-scheme: it is simply
a morphism in the category of $k$-formal schemes. Such a morphism is nothing but
a morphism in the category of $k$-locally ringed spaces (the continuity
conditions are automaticaly satisfied, since the topological sheaf of ring of a
scheme interpreted as a formal scheme is discrete). 
In turn, a $k$-morphism from the completion $\hX_Y$ to the $k$-scheme $T$ may be
identified with a compatible system of morphisms of $k$-schemes from the
successive infinitesimal neighborhoods $Y_i$ of $Y$ in $X$ to $T$.}.

Let us assume that $Y$  is a local complete intersection contained in the
regular locus $X_{\rm reg}$ of $X$.

1) If the normal bundle $N_Y X$ of $Y$ in $X$ is ample on $Y$ and if every
component of $Y$ has dimension $\geq 1,$ then there exist an \'etale
neighborhood of $Y$ in $X_{\rm reg}$
\begin{equation}\label{etnbdreg}
\xymatrix{ & \tilde{X}\ar[d]^\nu \\
Y \ar[ur]^j\ar[r]^i & X_{\rm reg},}
\end{equation}
(where $i$ denotes the inclusion morphism) and a morphism of $k$-schemes 
$$\phi: \tilde{X} \lra T$$
such that the $k$-morphisms of formal schemes induced by $\nu$ and $\phi$, 
$$\hat{\nu}: \widehat{\tilde{X}}_{j(Y)} \lrasim \hX_Y$$
and 
$$\hat{\phi}: \widehat{\tilde{X}}_{j(Y)} \lra T,$$ make the following diagram
commutative:
\begin{equation}\label{commnor}
\xymatrix{ 
\widehat{\tilde{X}}_{j(Y)}\ar[d]^{\hat{\nu}}_{\simeq}\ar[rd]^{\hat{\phi}} & 
 \\
{\hX_Y}\ar[r]^\psi & T.}
\end{equation}

2) If there exist some ample effective Cartier divisors $H_1,\ldots,H_e$, $1\leq
e \leq d-1$, in $X$ such that $Y$ has dimension $d-e$ and may be written as the
complete intersection 
$$Y =  H_1 \cap \cdots \cap H_e,$$
then there exists an open Zariski neighborhood $U$ of $Y$ in $X$ and a morphism
of $k$-schemes 
$$\phi : U \lra T$$
such that the morphism of $k$-formal schemes induced by $\phi$,
$$\hat{\phi}: \hX_Y = \widehat{U}_Y \lra T ,$$
coincides with $\psi.$
\end{theorem}

\begin{proof} Let us assume that the hypotheses of 1) are satisfied. Then we may
introduce the $k$-scheme
$$Z:= X\times_k T,$$
and consider $Y$ as a subscheme of $Z$, by means of the closed immersion
$$i':=(i,\psi_{\vert Y}): Y \lra Z.$$

The graph of $\psi$, seen as a $k$-morphism of formal schemes from $\hX_Y$ to
$\hT_{\psi(Y)}$, defines a closed formal subscheme $\hV$ of $\hZ_{i'(Y)}$
admitting $i'(Y)$ as scheme of definition. The projections
$$\pr_1: Z \lra X  \;\mbox{ and }\;  \pr_2: Z \lra T $$
define, after completion, morphisms of formal schemes
$$\widehat{\pr}_1: \hZ_{i'(Y)} \lra \hX_Y  \;\mbox{ and }\; \widehat{\pr}_2:
\hZ_{i'(Y)} \lra \hT_{\psi(Y)}.$$
By restriction to $\hV$, they define an isomorphism of formal schemes
$$\widehat{\pr}_{1\vert \hV}: \hV \lrasim \hX_Y$$
and a morphism
$$\widehat{\pr}_{2\vert \hV}: \hV \lra \hT_{\psi(Y)}$$
such that 
$$\widehat{\pr}_{2\vert \hV}=
\psi \circ \widehat{\pr}_{1\vert \hV}.$$
Clearly the restriction of $\widehat{\pr}_{1\vert \hV}$ to the scheme of
definition $i'(Y)$ of $\hV$ is the first projection
\begin{equation}\label{prY}
\pr_{1\vert i'(Y)}: i'(Y) \lrasim Y.
\end{equation}

Consequently $\hV$ is smooth and $i'(Y)$ is a l.c.i. in $\hV$, and through the
isomorphism (\ref{prY}), the normal bundle $N_{i'(Y)}\hV$ gets identified to
$N_Y{\hX_Y} \simeq N_YX,$ hence is ample.

According to Theorem \ref{thalgnrample}, the formal subscheme $\hV$ is
algebraic. By Proposition \ref{algequ}, there exists a commutative diagram of
$k$-schemes
$$
\xymatrix{ & M \ar[d]^\mu\\ 
Y\ar[r]_{(i,\psi_{\vert Y})\;\;\;\;\;} \ar[ur]^j &  X\times_k T}
$$
with $M$ smooth over $k$, such that, after completing along $j(Y)$, $\mu$
becomes an isomorphism:
$$\hat{\mu}: \hM_{j(Y)} \lrasim \hV \,(\hra \hZ_{i'(Y)}).$$

Let us consider the open subscheme of $M$:
$$\tilde{X} := (\pr_1 \circ \mu)^{-1}(X_{\rm reg}).$$
It is straightforward that $j(Y)$ lies in $\tilde{X}$ and that the diagram
$$
\xymatrix{ & \tilde{X}\ar[d]^{\nu:=\pr_1\circ \mu} \\
Y \ar[ur]^j\ar[r]^i & X_{\rm reg}}
$$
is an \'etale neighborhood of $Y$ in $X_{\rm reg}$ which satisfies the
conclusion of 1).

Let us now assume that the hypotheses of 2) are satisfied. Then the normal
bundle of $X$ in $Y$ is isomorphic to
$\bigoplus_{j=1}^e \cO_X(H_j)_{\vert Y},$ hence is ample. According to Part 1)
of the Theorem, there exists an \'etale neighborhood (\ref{etnbdreg}) of $Y$ in
$X$ and a morphism of $k$-schemes $\phi: \tilde{X} \ra Z$ such that the diagram
(\ref{commnor}) is commutative. Proposition \ref{etZar} shows that, after
possibly shrinking $\tilde{X},$ we may assume that $\nu$ is an open immersion.
It may be identified with the inclusion morphism in $X_{\rm reg}$ of some open
Zariski neighborhood $U$ of $Y$, and the commutativity of  (\ref{commnor})
precisely asserts that $\hat{\phi}= \psi.$

\end{proof}

\subsection{Algebraization of formal bundles and subbundles}\label{secalgbund}

In this Section, we place ourselves under the assumptions of Theorem
\ref{algmor}, 2).  Namely, we denote by $X$  a projective $k$-scheme of
dimension $d$, by $e$ a positive integer $\leq d-1$, and by $H_1,\ldots,H_e$
ample effective Cartier divisors in $X$ such that 
$$Y:=  H_1 \cap \cdots \cap H_e$$
has (necessarily pure) dimension $d-e$ and is contained in the regular locus
$X_{\rm reg}$ of $X$. 

Besides, we choose an ample line bundle $\cO_X(1)$ over $X$. 

We shall say that a vector bundle $\hE$ over $\hX_Y$ is \emph{algebraizable} if
there exists a coherent sheaf of $\cO_X$-modules $E$ over $X$ (necessarily
locally free on some open Zariski neighborhood of $Y$)  such that $\hE \simeq
E_{\vert \hX_Y}$, or equivalently, if there exists a vector bundle on some open
Zariski neighborhood  of $Y$ in $X$ such that $\hE \simeq E_{\vert \hX_Y}.$

The category of algebraizable vector bundles over $\hX_Y$ is clearly stable
under elementary tensor operations, like directs sums, tensor products, and
dual.  

\begin{theorem}\label{Lef} 1) For any formal vector bundle $\hE$ over $\hX_Y$,
the space of global sections $\Gamma (\hX_Y, \hE)$ is a finite dimensional
$k$-vector space.

2) For any open Zariski neighborhood $V$ of $Y$ in $X$ included in the open
subscheme $X_{\rm nor}$ of normal points of $X$ and any vector bundle $E$ over
$V,$ the restriction map
\begin{equation}\label{isomEV}
\Gamma(V,E) \lra \Gamma (\hX_Y, E_{\vert \hX_Y})
\end{equation}
is an isomorphism.
\end{theorem}

 Part 1) of Theorem \ref{Lef} is a special case of \cite{Hartshorne68}, Theorem
6.2. Actually, we will not really use this result in this Appendix.
 
 In the terminology of  \cite{GrothendieckSGA2}, X.2, and  \cite{Hartshorne70},
IV.1, Part 2) asserts that the pair $(X_{\rm nor}, Y)$ satisfies the
\emph{Lefschetz property} $\mathbf{Lef}(X_{\rm nor}, Y).$ When $e=1$, it is a
special case of \cite{GrothendieckSGA2}, XII.2, Corollaire\footnote{Actually,
the statement of Corollaire 2.4 in \emph{loc. cit.} does not exactly cover the
situation we deal with here, but require slightly stronger hypotheses, satisfied
for instance when $X$ itself is normal.  By working on the normalization of $X$,
one may actually assume that these hypotheses are satisfied.} 2.4. When  $X$ is
smooth over $k$, it is proved in \cite{Hartshorne70}, X.1, Corollary 1.2.
\begin{proof} We are left to prove 2) in our setting. It will follow from the
the algebraization criterion for formal morphisms established in the previous
section.

Let us consider $V$ and $E$ as in 2), and let us introduce the ``total space"
$$\V(E^\vee) := \Spec_V {\rm Sym} E^\vee$$
of the vector bundle $E$, and its structural morphism
$$p:   \V(E^\vee) \lra V.$$

For any open subscheme $U$ of $V$, the elements of $\Gamma(U,E)$ may be
identified with the morphisms of $k$-schemes 
$$s: U \lra \V(E^\vee)$$
which are sections of $p$ over $U$ (that is, which satisfy  $p\circ s = Id_U$).
Similarly, the elements of $\Gamma (\hX_Y,  E_{\vert \hX_Y})$ may be 
identified with the $k$-morphisms of formal schemes 
$$t: \hX_Y \lra  \V(E^\vee)$$
which are sections of $p$ over $\hX_Y.$ 

Observe also that, for any morphism of $k$-schemes
$$\phi: U \lra \hV(E^\vee) $$
defined on some open Zariski neighborhood of $U$ of $Y$ in $X,$ if the induced
morphism of formal schemes
$$\hat{\phi}: \hX_Y = \hV_Y \lra \V(E^\vee)$$  
is a section  of $p$ over $\hX_Y$, then $\phi$ is a section of $p$ over $U$
(indeed $Y$ is non-empty, and $X$ --- hence $U$ --- is an integral scheme).

The injectivity of the restriction morphism (\ref{isomEV}) is straightforward.
Let us show that it is surjective.

Together with the remarks above, Theorem \ref{algmor}, 2) applied with $T:=
\V(E^\vee)$ establishes that, for any formal section $t$ in $\Gamma (\hX_Y, 
E_{\vert \hX_Y})$, there exists an open Zariski neighborhood $U$ of $Y$ in $V$
and a section $s$ in $\Gamma(U, E)$ such that 
$$s_{\vert \hX_Y} = t.$$

 Observe that the dimension $\dim I$ of  any closed integral subscheme $I$ of
$X$ disjoint of $Y$ satisfies
 $$\dim I < e.$$
 (Otherwise the intersection number 
 $$\cO_X(1)^{\dim I -e}. H_1.\ldots.H_e.I$$
 would vanish, in contradiction to the ampleness of $\cO_X(1)$ and of
$H_1,\ldots, H_e.$)
 Consequently the codimension in $X$ of any component of $X\setminus U$ is at
least $2$. This implies that the depth of $\cO_X$ at every point of $X_{\rm nor}
\setminus U$ is at least $2$ and that the restriction morphism
$$\Gamma (V,E) \lra \Gamma(U,E)$$
is an isomorphism. 

This shows that the section $s$ of $E$ over $U$ uniquely extends to a section
over $V$, and completes the proof.

\end{proof}

The following two theorems are closely related --- each of them may easily be
deduced from the other one --- and will be established together.

\begin{theorem}\label{algebglobsec} For any vector bundle $\hE$ over $\hX_Y$,
the following conditions are equivalent:
\begin{enumerate}
\item The vector bundle $\hE$ is algebraizable.
\item For any large enough positive integer $D$, the vector bundle  over 
$\hX_Y$
$$ \hE(D):=\hE\otimes_{\cO_X}\cO_X(D)$$  is generated by its global sections
over $\hX_Y$.
\item For some integer $D$, the vector bundle  $\hE(D)$ is generated by its
global sections over $\hX_Y$.
\end{enumerate}
\end{theorem}

 With the notation of Theorem \ref{algebglobsec}, observe that a morphism of
coherent 
$\cO_{\hX_Y}$-modules $\hat{\varphi}: \hat{\cF} \lra \hat{\cG}$ is onto iff its
restriction $\hat{\varphi}_{\mid Y}: \hat{\cF}_{\mid Y} \lra \hat{\cG}_{\mid Y}$
to $Y$ is onto (this directly follows from Nakayama's Lemma). Consequently, the
vector bundle $\hE(D)$ is generated by its global sections over $\hX_Y$ iff the
vector bundle over $Y$
$$\hE(D)_{\vert Y} \simeq \hE_{\vert Y} \otimes_{\cO_X} \cO_X(D)$$
is generated by its global sections in the image of the restriction morphism
\begin{equation}\label{surYfor}
\Gamma(\hX_Y, \hE(D)) \lra \Gamma (Y, \hE(D)_{\vert Y}).
\end{equation}

\begin{theorem}\label{subquot} If  a vector bundle $\hE$ over $\hX_Y$ is
algebraizable, then any quotient vector bundle and any sub-vector
bundle\footnote{Namely, any subsheaf of $\cO_{\hX_Y}$-modules which is locally a
direct summand.} of $\hE$ is algebraizable.
\end{theorem}

\begin{proof} We begin by the proof of Theorem \ref{subquot}.

Let $E$ be a vector bundle over some open Zariski neighborhood $V$ of $Y$ in
$X$. We are going to show that any quotient bundle of $E_{\hX_Y}$ over $\hX_Y$
is algebraizable. By a straightforward duality argument, this will also prove
that any sub-vector bundle of an algebraizable vector bundle over $\hX_Y$ is
algebraizable. 

Let $$p: \mathbf{Gr} := \coprod_{0\leq n \leq \rk E} \mathbf{Grass}_n (E) \lra
V$$ be the Grasmannian scheme of the vector bundle $E$ over $V$, which
classifies locally free quotients of $E$ over varying $V$-schemes; \cf
\cite{EGAI}, Section 9.7. (It is a ``Zariski locally trivial bundle over $V$",
with a fiber isomorphic to the disjoint union of the classical Grassmann
varieties, classifying quotients of rank $n$  of $k^{\rk E}$, for $0\leq n \leq
\rk V.$)

The vector bundle $p^{\ast} E$ over $\mathbf{Gr}$ is equipped with a canonical
quotient bundle
$$\mathbf{q} : p^\ast E \lra \mathbf{Q},$$ and, by the very construction of
$\mathbf{Gr}$, for any subscheme $S$ of $X$, the map which sends a section
$\sigma$ of 
$$p_S: p^{-1}(S) \lra S$$ to the quotient vector bundle 
$$\sigma^\ast \mathbf{q} : E_{\vert S} \simeq \sigma^\ast p^\ast E \lra
\sigma^\ast \mathbf{Q}$$
establishes a bijection between the set of sections of $p_S$ and the set of
(isomorphisms classes) of quotient vector bundles of $E_{\vert S}$ over $S$.

By taking for $S$ the successive thickenings $Y_i$ of $Y$ in $X$, this bijective
correspondence extends between (formal) sections of $p$ over $\hX_Y$ and
quotient vector bundles of $\hE :=E_{\vert \hX_Y}$. In other words, for any
quotient vector bundle
\begin{equation}\label{quotfor}
\hat{q}:\hE \lra \widehat{Q},
\end{equation}
we may consider its ``classifying map" 
$$\psi: \hX_Y \lra \mathbf{Gr},$$
which is a section of $p$ over $\hX_Y$: the quotient vector bundle 
$$\psi^\ast \mathbf{q}: E_{\vert\hX_Y} \lra \psi^\ast \mathbf{Q}$$
is isomorphic to  (\ref{quotfor}).

According to Theorem \ref{algmor} applied to $T=\mathbf{Gr}$, the formal
morphism $\psi$ is induced by some morphism of $k$-schemes
$$\phi : U \lra \mathbf{Gr}$$
defined on some open Zariski neighborhood $U$ of $Y$ in $X$, which we may assume
to lie in $V$. 
It is straightforward that $\phi$, like $\psi$, is a section of $p$. Moreover
the corresponding quotient bundle 
$$\phi^\ast \mathbf{q}: E_{\vert U} \lra \phi^\ast \mathbf{Q}$$
becomes isomorphic to (\ref{quotfor}) after restriction to $\hX_Y$ (that is,
after completing along $Y$).  

In particular, $\widehat{Q}$ is isomorphic to the restriction to $\hX_Y$ of the
vector bundle $\phi^\ast \mathbf{Q}$ on $V$, and is therefore algebraizable.

We now turn to the proof of Theorem  \ref{algebglobsec}. 

To prove the implication $(1)\Rightarrow (2),$ observe that, for any
coherent sheaf of $\cO_X$-modules $E$ over $X$  such that $\hE \simeq E_{\vert
\hX_Y},$ 
the image of the restriction morphism 
\begin{equation}\label{surY}
 \Gamma(X, E(D)) \lra \Gamma (Y, E(D)_{\vert Y}) \simeq  \Gamma (Y,
\hE(D)_{\vert Y})
 \end{equation}
is contained in the image of (\ref{surYfor}), and that, for any large  enough
positive integer $D$, the  morphism (\ref{surY}) is surjective and  $E(D)_{\vert
Y}$ is generated by its global sections over $Y$.

The implication $(2)\Rightarrow (3)$ is trivial. 

Finally, assume that Condition (3) is satisfied, and consider the
``tautological" morphism of vector bundles over $\hX_Y$:
\begin{equation}\label{ponto}
p: \Gamma(\hX_Y, \hE(D))\otimes_k \cO_{\hX_Y} \lra \hE(D). 
\end{equation}
(Recall that, according to Theorem \ref{Lef}, 1), the $k$-vector space
$\Gamma(\hX_Y, \hE(D))$ is finite dimensional. We could easily avoid to rely on
this result by replacing $\Gamma(\hX_Y, \hE(D))$ by a ``sufficiently large"
finite dimensional sub-vector space.) By hypothesis, it is surjective, and
$\hE(D)$ is therefore identified with a quotient of the ``trivial" vector bundle
$\Gamma(\hX_Y, \hE(D))\otimes_k \cO_{\hX_Y}$ over $\hX_Y$, which is clearly
algebraizable. According to Theorem \ref{subquot}, $\hE(D)$ is therefore
algebraizable. Finally, the vector bundle
$$\hE \simeq \hE(D) \otimes_{\cO_X} \cO_{X}(-D)$$
is algebraizable. This completes the proof of  $(3)\Rightarrow (1)$.

\end{proof}

\subsection{Comments and examples}

\subsubsection{} Under the assumptions of Theorem \ref{algmor}, Part 1), the
conclusion of Part 2) does \emph{not} hold in general. In other words, under the
mere assumption of ampleness of the normal bundle $N_Y X$, the \'etale
neighborhood of $Y$ in $X$ onto which $\psi$ extends cannot be chosen to be a
Zariski neighborhood.

This is demonstrated by a classical example of Hironaka (see \cite{Hironaka68},
p. 588, and \cite{Hartshorne70}, Chapter V), which in its simplest form may be
presented as follows.

Consider a smooth, connected, projective threefold over $k$ whose algebraic
fundamental group is not trivial, and choose  a finite \'etale covering
$$\nu: \tilde{X} \lra X,$$
with $\tilde{X}$ connected and (necessarily) projective. Define $\tilde{Y}$ as a
``general" intersection of two effective divisors in the linear system defined
by some large multiple of an ample divisor on $\tilde{X}$.  Then $\tilde{Y}$ is
smooth, connected, its normal bundle $N_{\tilde{Y}}\tilde{X}$ is ample, and the
map
$$\nu_{\vert \tilde{Y}} \tilde{Y} \lra Y:=\nu(\tilde{Y})$$
is an isomorphism. 

Since $\nu$ is \'etale, it induces an isomorphism between completions
$$\hat{\nu}: \widehat{\tilde{X}}_{\tilde{Y}} \lrasim \hX_Y$$
and an isomorphism of normal vector bundles:
$$N_{\tilde{Y}}\tilde{X} \lrasim \nu_{\vert \tilde{Y}}^\ast N_Y X.$$
In particular, $N_Y X$ also is ample.

If we let 
$$T:= \tilde{X} \mbox{\;\;\;and \;\;\;} \psi:= \hat{\nu}^{-1}: \hX_Y \lra
\widehat{\tilde{X}}_{\tilde{Y}} \hra \tilde{X},$$
then it is straightforward that $\psi$ cannot be realized as the restriction to
$\hX_Y$ of some morphism of $k$-schemes from some open Zariski neighborhood of
$Y$ in $X$ to $\tilde{X}.$

\subsubsection{} The results of Section \ref{secalgbund} may be rephrased in
terms of functors between categories of vector bundles.  

Let indeed $V$ be an open Zariski neigborhood of $Y$ in $X_{\rm nor}$. We may
introduce the $k$-linear categories $\Bun (V)$ of vector bundles on $V$,
$\Bun(X_Y)$ of germs of vector bundles on $X$ along $Y$, and $\Bun(\hX_Y)$ of
vector bundles on $\hX_Y,$
and the obvious restriction (or completion) functors
$$\Bun(V) \stackrel{\digamma^V_{X_Y}}{\lra} \Bun (X_Y)
\stackrel{\digamma^{X_Y}_{\hX_Y}}{\lra}\Bun(\hX_Y).$$

Theorem \ref{Lef}, 2), and its proof show that the functors  $\digamma^V_{X_Y}$
and $\digamma^{X_Y}_{\hX_Y}$ are fully faithfull.

When $V$ is smooth, the functor $\digamma^V_{X_Y}$ is easily seen to be
essentially surjective iff $d=2$ (and consequently $e=1$).

The functor $\digamma^{X_Y}_{\hX_Y}$ is \emph{not} essentially surjective when
$$\dim Y (= d-e) =1.$$
Indeed, the following Proposition --- which is a straightforward consequence of
Theorem \ref{Lef}, 2) --- allows one to construct non-algebraizable formal
vector bundles on $\hX_Y$ when $\dim Y = 1.$

\begin{proposition} Let us keep the notation of Section  \ref{secalgbund}, and
assume that $\dim Y = 1.$

Let $p$ be a point of $Y(k)$ and let $f$ be an element of the local ring
$\cO_{\hX_Y,P}$ such that $f(P) =0$ and $f_{\vert Y} \in \cO_{Y,P}$ is
invertible on $\Spec \cO_{Y,P} \setminus \{P\}.$ Then the following two
conditions are equivalent:
\begin{enumerate}
\item  The line bundle on $\hX_Y$ defined by the divisor $\div f$ is
algebraizable.
\item There exists an effective (Cartier) divisor  on $X_{\rm reg}$ whose
completion along $Y$ coincides with $\div f$.
\end{enumerate}
\end{proposition}

The essential surjectivity of the functor  $\digamma^{X_Y}_{\hX_Y}$ is precisely
the \emph{effective Lefschetz property} $\mathbf{Leff}(X,Y)$ considered in
\cite{GrothendieckSGA2}, X.2, and  \cite{Hartshorne70}, IV.1. According to
\cite{GrothendieckSGA2}, XII, Corollaire 3.4 and to \cite{Hartshorne70}, IV,
Theorem 1.5, it holds under the assumptions of Section \ref{secalgbund},  when
moreover $X$ is smooth, some positive multiples of the $H_j$'s lie in the same
linear system, and $\dim Y \geq 2$.

It appears very likely that, under the assumptions of Section \ref{secalgbund},
this last condition   $\dim Y \geq 2$ would be enough to ensure the validity of
$\mathbf{Leff}(X,Y)$.

This discussion shows that Theorem \ref{subquot}, which asserts the stability
under quotients of the essential image of $\digamma^{X_Y}_{\hX_Y}$, is
significant mostly when $\dim Y =1.$

\end{document}